\documentclass[sn-mathphys,Numbered]{sn-jnl}

\usepackage{graphicx}%
\usepackage{multirow}%
\usepackage{amsmath,amssymb,amsfonts}%
\usepackage{amsthm}%
\usepackage{mathrsfs}%
\usepackage[title]{appendix}%
\usepackage{xcolor}%
\usepackage{textcomp}%
\usepackage{manyfoot}%
\usepackage{booktabs}%
\usepackage{algorithm}%
\usepackage{algorithmicx}%
\usepackage{algpseudocode}%
\usepackage{listings}%

\usepackage[LGR,T1]{fontenc}
\usepackage[latin9]{inputenc}
\usepackage{color}
\usepackage{mathtools}

\theoremstyle{plain}
\newtheorem{thm}{\protect\theoremname}
\theoremstyle{definition}
\newtheorem{defn}{\protect\definitionname}
\theoremstyle{plain}
\newtheorem{conjecture}{\protect\conjecturename}
\theoremstyle{plain}
\newtheorem{lem}{\protect\lemmaname}
\theoremstyle{plain}
\newtheorem{prop}{\protect\propositionname}
\theoremstyle{plain}
\newtheorem{cor}{\protect\corollaryname}
\theoremstyle{remark}
\newtheorem{rem}{\protect\remarkname}

\usepackage{enumerate}
\usepackage{verbatim}
\usepackage{mathrsfs}
\usepackage{centernot}

\mathtoolsset{showonlyrefs} 

\providecommand{\keywords}[1]{\textbf{Index terms---} #1}

\usepackage{tikz}
\usetikzlibrary{patterns,arrows}

\date{\today}

\def\E#1{\mathbb{E}}

  \providecommand{\definitionname}{Definition}
  \providecommand{\lemmaname}{Lemma}
  \providecommand{\remarkname}{Remark}
\providecommand{\theoremname}{Theorem}

\allowdisplaybreaks[1]
\def\Ent{\operatorname{Ent}}

\providecommand{\conjecturename}{Conjecture}

\providecommand{\corollaryname}{Corollary}

\providecommand{\propositionname}{Proposition}

\usepackage{subfig}
\makeatother

\providecommand{\conjecturename}{Conjecture}
\providecommand{\corollaryname}{Corollary}
\providecommand{\definitionname}{Definition}
\providecommand{\lemmaname}{Lemma}
\providecommand{\propositionname}{Proposition}
\providecommand{\remarkname}{Remark}
\providecommand{\theoremname}{Theorem}

\theoremstyle{thmstyleone}%
\theoremstyle{thmstyletwo}%

\theoremstyle{thmstylethree}%

\begin{document}

\title[On the $\Phi$-Stability and Related Conjectures]{On the $\Phi$-Stability and Related Conjectures\footnote{This arXiv version contains more proof details than the  version published on Probability Theory and Related Fields.}}


\author*[1]{\fnm{Lei} \sur{Yu}}\email{leiyu@nankai.edu.cn} 

\affil*[1]{\orgdiv{School of Statistics and Data Science, LPMC, KLMDASR,
and LEBPS}, \orgname{Nankai University}, \orgaddress{\street{94 Weijin Rd}, \city{Tianjin}, \postcode{300071}, \state{Tianjin}, \country{China}}} 
 

\abstract{Given a convex function $\Phi:[0,1]\to\mathbb{R}$ and the mean $\mathbb{E}f(\mathbf{X})=a\in[0,1]$,
which Boolean function $f$ maximizes the $\Phi$-stability $\mathbb{E}[\Phi(T_{\rho}f(\mathbf{X}))]$
of $f$? Here $\mathbf{X}$ is a random vector uniformly distributed
on the discrete cube $\{-1,1\}^{n}$ and $T_{\rho}$ is the Bonami-Beckner
operator.  Special cases of this problem include the (symmetric and
asymmetric) $\alpha$-stability problems and the ``Most Informative
Boolean Function'' problem. In this paper, we provide several upper
bounds for the maximal $\Phi$-stability. When specializing $\Phi$
to some particular forms, by these upper bounds, we partially resolve
Mossel and O'Donnell's conjecture on $\alpha$-stability with $\alpha>2$,
Li and Médard's conjecture on $\alpha$-stability with $1<\alpha<2$,
and Courtade and Kumar's conjecture on the ``Most Informative Boolean
Function'' which corresponds to a conjecture on $\alpha$-stability
with $\alpha=1$. Our proofs are based on discrete Fourier analysis,
optimization theory, and improvements of the Friedgut--Kalai--Naor
(FKN) theorem. Our improvements of the FKN theorem are sharp or asymptotically
sharp for certain cases. }




\maketitle

\section{\label{sec:Introduction}Introduction}

Let $\mathbf{X}$ be a random vector uniformly distributed on the
discrete cube $\{-1,1\}^{n}$. Let $\mathbf{Y}\in\{-1,1\}^{n}$ be
the random vector obtained by independently changing the sign of each
component of $\mathbf{X}$ with the same probability $\frac{1-\rho}{2}$.
Here, $\rho\in[0,1]$ corresponds to the correlation coefficient between
each component of $\mathbf{X}$ and the corresponding one of $\mathbf{Y}$.
Let $T_{\rho}$ be the noise operator (or the \emph{Bonami-Beckner
operator}) which acts on Boolean functions $f:\{-1,1\}^{n}\to\{0,1\}$
such that $T_{\rho}f(\mathbf{x})=\mathbb{E}[f(\mathbf{Y})|\mathbf{X}=\mathbf{x}]$.
Let $\Phi:[0,1]\to\mathbb{R}$ be a continuous and strictly convex
function. 
\begin{defn}
For a Boolean function $f:\{-1,1\}^{n}\to\{0,1\}$, the \emph{$\Phi$-stability}
of $f$ with respect to (w.r.t.) correlation coefficient $\rho$ is
defined as 
\begin{align*}
\mathbf{Stab}_{\Phi}[f] & =\mathbb{E}[\Phi(T_{\rho}f(\mathbf{X}))].
\end{align*}
\end{defn}

The $\Phi$-stability is closely related to the $\Phi$-entropy and
the $\Phi$-mutual information. Define the \emph{$\Phi$-entropy}
of $f(\mathbf{Y})$ as 
\[
\Ent_{\Phi}(f(\mathbf{Y})):=\mathbb{E}[\Phi(f(\mathbf{Y}))]-\Phi(\mathbb{E}f(\mathbf{Y})),
\]
the \emph{conditional $\Phi$-entropy} of $f(\mathbf{Y})$ given $\mathbf{X}$
as 
\begin{align*}
\Ent_{\Phi}(f(\mathbf{Y})|\mathbf{X}) & :=\mathbb{E}[\Phi(f(\mathbf{Y}))]-\mathbb{E}_{\mathbf{X}}\Phi(\mathbb{E}[f(\mathbf{Y})|\mathbf{X}])\\
 & =\mathbb{E}[\Phi(f(\mathbf{Y}))]-\mathbf{Stab}_{\Phi}[f],
\end{align*}
and the \emph{$\Phi$-mutual information} from $f(\mathbf{Y})$ to
$\mathbf{X}$ as 
\begin{align}
I_{\Phi}(f(\mathbf{Y});\mathbf{X}) & :=\Ent_{\Phi}(f(\mathbf{Y}))-\Ent_{\Phi}(f(\mathbf{Y})|\mathbf{X})\nonumber \\
 & =\mathbf{Stab}_{\Phi}[f]-\Phi(\mathbb{E}f(\mathbf{Y}))=\Ent_{\Phi}(T_{\rho}f(\mathbf{X})).\label{eq:-72}
\end{align}
Hence, given the expectation of $f$, $\mathbf{Stab}_{\Phi}[f]$ is
equal to the $\Phi$-mutual information from $f(\mathbf{Y})$ to $\mathbf{X}$,
up to a term which only depends on the expectation of $f$. In addition,
observe that given the expectation of $f$, when there is no noise
corruption (i.e., $\rho=1$), $\mathbf{Stab}_{\Phi}[f]=\mathbb{E}[\Phi(f(\mathbf{X}))]$
which is also fixed. Hence, $\mathbf{Stab}_{\Phi}[f]$ also quantifies
the change of the $\Phi$-entropy of $f$ after taking the noise operator.

In fact, the $\Phi$-entropy is more related to the relative entropy
than the Shannon entropy. Any nonnegative function $g$ such that
$\mathbb{E}_{\mathbf{Z}\sim P}[g(\mathbf{Z})]=1$ can be seen as the
Radon--Nikodym derivative of a probability distribution $Q$ w.r.t.
another probability distribution $P$, $\Ent_{\Phi}(g)$ corresponds
to a generalized relative entropy (called the ``$f$-divergence'')
from $Q$ to $P$ \cite{csiszar1964informationstheoretische,csiszar1967information,ali1966general}.

The noise stability problem, in a general sense, concerns which Boolean
functions (or measurable sets) are the ``most stable'' under the
action of the noise operator. In terms of $\Phi$-stability, the noise
stability problem is formulated as follows. 
\begin{defn}
The\emph{ maximal $\Phi$-stability} w.r.t. volume $a\in[0,1]$ is
defined as 
\begin{align*}
\mathbf{MaxStab}_{\Phi}(a) & =\max_{f:\{-1,1\}^{n}\to\{0,1\},\mathbb{E}f=a}\mathbf{Stab}_{\Phi}[f].
\end{align*}
\end{defn}

By the relation in \eqref{eq:-72}, determining the maximal $\Phi$-stability
w.r.t. volume $a\in[0,1]$ is equivalent to determining the maximum
$I_{\Phi}(f(\mathbf{Y});\mathbf{X})$ over all Boolean functions with
$\mathbb{E}f=a$.

We next consider two common instances of $\Phi$. For $\alpha\ge1$,
define $\Phi_{\alpha},\Phi_{\alpha}^{\mathrm{sym}}:[0,1]\to\mathbb{R}$
as 
\begin{align*}
\Phi_{\alpha}(t) & :=t\ln_{\alpha}(t)\quad\textrm{ and \quad}\Phi_{\alpha}^{\mathrm{sym}}(t):=t\ln_{\alpha}(t)+(1-t)\ln_{\alpha}(1-t).
\end{align*}
where $0\ln_{\alpha}(0):=0$, and the function $\ln_{\alpha}:(0,+\infty)\to\mathbb{R}$
for $\alpha\in\mathbb{R}$ is defined as 
\[
\ln_{\alpha}(t):=\begin{cases}
\frac{t^{\alpha-1}-1}{\alpha-1}, & \alpha\neq1\\
\ln(t), & \alpha=1
\end{cases}
\]
and is known as the \emph{$\alpha$-logarithm} (or $q$\emph{-logarithm})
introduced by Tsallis \cite{tsallis1994numbers}, but with a slight
reparameterization. Note that $\Phi_{\alpha}(t)$ and $\Phi_{\alpha}^{\mathrm{sym}}(t)$
are non-decreasing in $\alpha$ given $t\in[0,1]$ (since so is $\ln_{\alpha}(t)$),
and $\Phi_{1}$ and $\Phi_{1}^{\mathrm{sym}}$ (with $\alpha=1$)
is the continuous extension of the ones with $\alpha>1$. For brevity,
we term the maximal $\Phi_{\alpha}$-stability and the maximal $\Phi_{\alpha}^{\mathrm{sym}}$-stability
respectively as the \emph{maximal asymmetric and symmetric $\alpha$-stabilities}\footnote{The concept ``$\alpha$-stability'' for real $\alpha$ was introduced
in \cite{eldan2015two} for Gaussian distributions and in \cite{li2019boolean}
for binary distributions. The special case with $\alpha$ being a
positive integer was introduced in \cite{mossel2005coin}. More precisely,
the ``$\alpha$-stability'' defined there refers to the $\Phi$-stability
with $\Phi:t\mapsto t^{\alpha}$ for the asymmetric version and $\Phi:t\mapsto t^{\alpha}+(1-t)^{\alpha}$
for the symmetric version. Hence, our definition is a linear transform
version of theirs. We abuse the term ``$\alpha$-stability'' here,
since when the mean $a$ is given, computing the maximal $\Phi$-stabilities
defined in \cite{li2019boolean} is equivalent to computing the maximal
$\Phi$-stabilities defined here. }, denoted as $\mathbf{MaxStab}_{\alpha}$ and $\mathbf{MaxStab}_{\alpha}^{\mathrm{sym}}$.
Similarly, the $\Phi_{\alpha}$-entropy and $\Phi_{\alpha}^{\mathrm{sym}}$-entropy
are respectively termed as the\emph{ asymmetric and symmetric $\alpha$-entropies},
which are respectively denoted as $\Ent_{\alpha}$ and $\Ent_{\alpha}^{\mathrm{sym}}$.
Similar conventions also apply to the $\Phi_{\alpha}$-mutual information
and $\Phi_{\alpha}^{\mathrm{sym}}$-mutual information, denoted as
$I_{\alpha}$ and $I_{\alpha}^{\mathrm{sym}}$. For Boolean $f$ and
for $\alpha=1$, by definition, $\Ent_{1}^{\mathrm{sym}}(f)$, $\Ent_{1}^{\mathrm{sym}}(f(\mathbf{Y})|\mathbf{X})$,
and $I_{1}^{\mathrm{sym}}(f(\mathbf{Y});\mathbf{X})$ respectively
reduce to the \emph{Shannon entropy} of $f(\mathbf{Y})$, the \emph{conditional
Shannon entropy} of $f(\mathbf{Y})$ given $\mathbf{X}$, and the
\emph{Shannon mutual information} (denoted as $I(f(\mathbf{Y});\mathbf{X})$)
between $f(\mathbf{Y})$ and $\mathbf{X}$. Furthermore, given $f$,
$\Ent_{\alpha}(f)$ and $\Ent_{\alpha}^{\mathrm{sym}}(f)$ are nondecreasing
in $\alpha$; see \cite[Theorem 33]{sason2016f}. Moreover, it is
easily seen that if \emph{dictator functions} maximize the asymmetric
$\alpha$-stability over all balanced Boolean functions, then they
also maximize the symmetric $\alpha$-stability. Here, dictator functions
refer to the functions $f_{\mathrm{d}}:=1\{x_{k}=1\}$ or $1\{x_{k}=-1\}$
for some $1\le k\le n$.

The study of the noise stability problem, or more precisely, a two-function
(or two-set) version of the noise stability problem called the \emph{non-interactive
correlation distillation (NICD) problem}, dates back to Gács and Körner's
and Witsenhausen's seminal papers \cite{gacs1973common,witsenhausen1975sequences}.
By utilizing the tensorization property of the maximal correlation,
Witsenhausen \cite{witsenhausen1975sequences} showed that for $\alpha=2$,
the asymmetric and symmetric $2$-stability w.r.t. $a=1/2$ are attained
by dictator functions. The symmetric $\alpha$-stability problem with
$\alpha\in\{3,4,5,...\}$ was studied by Mossel and O'Donnell \cite{mossel2005coin},
but only the case $\alpha=3$ was solved by them. For $\alpha=3$,
the maximal symmetric $\alpha$-stability w.r.t. $a=1/2$ is attained
by dictator functions, and moreover, Mossel and O'Donnell observed
that this is not true for $\alpha=10$. Mossel and O'Donnell conjectured\footnote{In fact, this is a stronger version of the original conjecture posed
by Mossel and O'Donnell. In their original version, the Boolean functions
are additionally restricted to be \emph{antisymmetric} (or \emph{odd}).} that dictator functions maximize the symmetric $\alpha$-stability
over all balanced Boolean functions for all $\alpha\in\{4,5,...,9\}$.
Since for $\alpha=1$, the symmetric $1$-mutual information is nothing
but the Shannon mutual information, the maximal $1$-stability problem
for this case can be interpreted as the problem of maximizing the
Shannon mutual information $I(f(\mathbf{Y});\mathbf{X})$ over all
Boolean functions $f$ with a given mean. In fact, this case was already
studied by Courtade and Kumar \cite{kumar2013boolean,courtade2014boolean},
but this problem still remains widely open, except for the extreme
cases. Courtade and Kumar conjectured\footnote{In fact, this is a weaker version of the original conjecture posed
by Courtade and Kumar. In their original version, the Boolean functions
are not restricted to be balanced.} that for the mean $a=1/2$, dictator functions maximize the symmetric
$1$-stability over all balanced Boolean functions. This conjecture
attracts lots of interest from different fileds \cite{anantharam2013on,kindler2015remarks,ordentlich2016improved,samorodnitsky2016entropy,pichler2018dictator,li2019boolean},
and it is regarded as one of the most fundamental conjectures at the
interface of information theory and the analysis of Boolean functions.
Along these lines, Li and Médard \cite{li2019boolean} conjectured
that for $\alpha\in(1,2)$, the maximal asymmetric $\alpha$-stability
is still attained by dictator functions. Here we unify and slightly
generalize the Mossel--O'Donnell conjecture, Courtade--Kumar conjecture,
and Li--Médard conjecture in the following two conjectures. 
\begin{conjecture}[Maximal Asymmetric $\alpha$-Stability Conjecture]
\label{conj:AsymmetricStability} For $0\le\rho\le1$ and $\alpha\in[1,9]$,
$\mathbf{MaxStab}_{\alpha}(\frac{1}{2})$ is attained by dictator
functions. 
\end{conjecture}

\begin{conjecture}[Maximal Symmetric $\alpha$-Stability Conjecture]
\label{conj:SymmetricStability} For $0\le\rho\le1$ and $\alpha\in[1,9]$,
$\mathbf{MaxStab}_{\alpha}^{\mathrm{sym}}(\frac{1}{2})$ is attained
by dictator functions. 
\end{conjecture}

Obviously, Conjecture \ref{conj:AsymmetricStability} implies Conjecture
\ref{conj:SymmetricStability}. Furthermore, as mentioned above, Conjectures
\ref{conj:AsymmetricStability} and \ref{conj:SymmetricStability}
with $\alpha=2$ were proven by Witsenhausen \cite{witsenhausen1975sequences};
Conjecture \ref{conj:SymmetricStability} with $\alpha=3$ was proven
by Mossel and O'Donnell \cite{mossel2005coin}; the Mossel--O'Donnell
conjecture corresponds to Conjecture \ref{conj:SymmetricStability}
with $\alpha\in\{4,5,...,9\}$; the Courtade--Kumar conjecture corresponds
to Conjecture \ref{conj:SymmetricStability} with $\alpha=1$; and
the Li--Médard conjecture corresponds to Conjecture \ref{conj:AsymmetricStability}
with $\alpha\in(1,2)$. Conjecture \ref{conj:AsymmetricStability}
is open for $\alpha\in[1,2)\cup(2,9]$, and Conjecture \ref{conj:SymmetricStability}
is open for $\alpha\in[1,2)\cup(2,3)\cup(3,9]$.

Recently, Barnes and Özgür \cite{barnes2020courtade} showed that
Conjecture \ref{conj:AsymmetricStability} with $\alpha=1$ and the
same one but with $\alpha\in(1,2)$ are equivalent, and Conjecture
\ref{conj:SymmetricStability} with $\alpha=1$ and the same one but
with $\alpha\in(1,2)$ are also equivalent. That is, the asymmetric
(resp. symmetric) version of Courtade--Kumar conjecture and the asymmetric
(resp. symmetric) version of Li--Médard conjecture are equivalent.
Following Barnes and Özgür's proofs in \cite[Subsections III.C and III.D]{barnes2020courtade},
one can obtain the following lemma, which slightly generalizes the
``only if'' parts of Barnes and Özgür's Theorems 1 and 2 in \cite{barnes2020courtade}. 
\begin{lem}
\label{lem:BarnesOzgur} For $a=1/2$, there are two thresholds $\alpha_{\min}$
and $\alpha_{\max}$ satisfying $1\le\alpha_{\min}\le2\le\alpha_{\max}$
such that dictator functions are optimal in attaining the asymmetric
max $\alpha$-stability with $\alpha\ge1$ if and only if $\alpha\in[\alpha_{\min},\alpha_{\max}]$.
This statement also holds for the symmetric max $\alpha$-stability
but with possibly different thresholds $\breve{\alpha}_{\min}$ and
$\breve{\alpha}_{\max}$ satisfying the same condition $1\le\breve{\alpha}_{\min}\le2\le\breve{\alpha}_{\max}$. 
\end{lem}

Here are some consequences of this lemma. Firstly, in terms of these
thresholds, Conjectures \ref{conj:AsymmetricStability} and \ref{conj:SymmetricStability}
can be restated as that $\alpha_{\min}=\breve{\alpha}_{\min}=1$ and
$\alpha_{\max},\breve{\alpha}_{\max}\ge9$. Secondly, as mentioned
previously, it was shown by Mossel and O'Donnell \cite{mossel2005coin}
that $\max\{\alpha_{\max},\breve{\alpha}_{\max}\}<10$. Lastly, Since
Conjecture \ref{conj:SymmetricStability} holds for $\alpha=3$ \cite{mossel2005coin},
we actually have $\breve{\alpha}_{\max}\ge3$, i.e., Conjecture \ref{conj:SymmetricStability}
holds for all $\alpha\in[2,3]$. In other words, Conjecture~\ref{conj:SymmetricStability}
is only open for $\alpha\in[1,2)\cup(3,9]$. Combining all these points
yields that $2\le\alpha_{\max}<10$ and $3\le\breve{\alpha}_{\max}<10$.
If Conjectures~\ref{conj:AsymmetricStability} and~\ref{conj:SymmetricStability}
are true, then the estimates of $\alpha_{\max}$ and $\breve{\alpha}_{\max}$
can be improved to $9\le\alpha_{\max},\breve{\alpha}_{\max}<10$.

\subsection{Our Contributions}

Our main contributions in this paper are: 
\begin{thm}[Formally Stated in Corollary \ref{cor:symmetricpower}]
\label{thm:symmetricpower}For $(\rho,\alpha)$ such that 
\[
0\le\rho\leq\begin{cases}
\rho^{*} & \alpha\in[1,2)\\
1 & \alpha\in[2,5]
\end{cases},
\]
$\mathbf{MaxStab}_{\alpha}^{\mathrm{sym}}(\frac{1}{2})$ is attained
by dictator functions, where $\rho^{*}\approx0.461491$ (the solution
to \eqref{eq:psi}). 
\end{thm}

\begin{thm}[Formally Stated in Corollary \ref{cor:asymmetricpower}]
\label{thm:asymmetricpower}For $(\rho,\alpha)$ such that 
\[
0\le\rho\leq\begin{cases}
\frac{1-\theta(\alpha)}{1+\theta(\alpha)} & \alpha\in(1,2)\\
1 & \alpha\in[2,3]
\end{cases},
\]
$\mathbf{MaxStab}_{\alpha}(\frac{1}{2})$ is attained by dictator
functions, where $\theta(\alpha)$ is some function (the solution
to \eqref{eq:theta-1}). 
\end{thm}

For symmetric $\alpha$-stability, our bound in Theorem \ref{thm:symmetricpower}
improves the previously best known bound $\breve{\alpha}_{\max}\ge3$
to $\breve{\alpha}_{\max}\ge5$, and for the asymmetric case, our
bound in Theorem \ref{thm:asymmetricpower} improves the previously
best known bound $\alpha_{\max}\ge2$ to $\alpha_{\max}\ge3$. In
other words, we have verified the Mossel--O'Donnell conjecture for
all $2\le\alpha\le5$ in the symmetric setting, and for all $2\le\alpha\le3$
in the asymmetric setting. Our results for $3<\alpha\le5$ in the
symmetric setting and for $2<\alpha\le3$ in the asymmetric setting
are new. As for the Courtade--Kumar conjecture and Li--Médard conjecture,
we improve Samorodnitsky's result \cite{samorodnitsky2016entropy}
in the sense that we provide an explicit dimension-independent threshold
$\rho^{*}\approx0.461491$ for which the symmetric versions of the
Courtade--Kumar and Li--Médard conjectures hold for all $\rho\in[0,\rho^{*}]$.

Theorems \ref{thm:symmetricpower} and \ref{thm:asymmetricpower}
are proven by combining discrete Fourier analysis and optimization
theory. In fact, they are consequences of a general bound on the maximal
$\Phi$-stability derived in this paper by using these techniques.
Furthermore, we also improve our bound for symmetric $\alpha$-stability
by incorporating improvements of Friedgut--Kalai--Naor (FKN) theorem
\cite{friedgut2002boolean} into our method. Our improved bounds are
presented in the optimization form, which seems difficult to simplify.
Numerical evaluation of this improved bound implies that the value
of the threshold $\rho^{*}$ in Theorem \ref{thm:symmetricpower}
can be improved to $0.83$. Our improvements of the FKN theorem are
sharp or asymptotically sharp for certain cases.

\subsection{Related Works}

We next summarize the literature on the noise stability problem. Although
the study of the noise stability problem originated in the seminal
papers \cite{gacs1973common,witsenhausen1975sequences,borell1985geometric},
the term \emph{noise stability} was first introduced by Benjamini,
Kalai, and Schramm \cite{benjamini1999noise}; see a brief survey
on \cite[p. 68]{O'Donnell14analysisof}. In \cite{benjamini1999noise},
only the case of $\alpha=2$ was studied. As mentioned previously,
this was subsequently generalized to the cases of $\alpha>2$, $\alpha=1$,
and $1<\alpha<2$ in \cite{mossel2005coin,courtade2014boolean,li2019boolean},
and the conjectures mentioned above were posed along with these generalizations.
In fact, a weaker version (the two-function version) of Courtade--Kumar
conjecture was solved by Pichler, Piantanida, and Matz \cite{pichler2018dictator}
by using Fourier analysis. Specifically, they showed that $I(f(\mathbf{Y});g(\mathbf{X}))$
is maximized by a pair of identical dictator functions over all Boolean
functions $(f,g)$, including but not limited to balanced Boolean
functions. In fact, if $(f,g)$ are additionally restricted to be
balanced, then this result is just a consequence of Witsenhausen's
maximal correlation bound \cite{witsenhausen1975sequences}. In other
words, Pichler, Piantanida, and Matz's contribution is addressing
the unbalanced case. However, the situation is totally different for
the single-function version of Courtade--Kumar conjecture, since
the latter is open even for the balanced case. A classic bound on
the mutual information $I(f(\mathbf{Y});\mathbf{X})$ for the balanced
case is $\rho^{2}$, proven by Witsenhausen and Wyner \cite{witsenhausen1975conditional}
(also see \cite{erkip1996efficiency}). Such a bound can be proved
via the so-called \emph{Mrs.\ Gerber's lemma}~\cite{wyner1973theorem}
or the \emph{hypercontractivity inequality} \cite{O'Donnell14analysisof}.
Ordentlich, Shayevitz, and Weinstein \cite{ordentlich2016improved}
improve this bound to a sharper one for small $\rho$. This new bound
turns out to be asymptotically sharp in the limiting case $\rho\to0$.
In 2016, Samorodnitsky \cite{samorodnitsky2016entropy} made a significant
breakthrough on the Courtade--Kumar conjecture. Specifically, he
proved the existence of a {\em dimension-independent} threshold
$\rho_{0}$ for which the Courtade--Kumar conjecture holds for all
$\rho\in[0,\rho_{0}]$. However, the value of $\rho_{0}$ was not
explicitly given in his paper, and required to be ``sufficiently
small''. A weaker version of this result in which the threshold $\rho_{0}$
is replaced by a sequence that vanishes as $n\to\infty$ was also
proven in \cite{ordentlich2016improved,yang2019most}. In addition,
by considering a variant of noise model, Eldan, Mikulincer, and Raghavendra
recently prove a variant version of the Courtade--Kumar conjecture
\cite{eldan2022noise}. Their proof is based on the so-called renormalized
Brownian motion.

The noise stability of \emph{unbalanced} Boolean functions was also
widely investigated in the literature. For the mean $a=1/4$, by combining
Fourier analysis with a coding-theoretic result, the present author
and Tan \cite{yu2021non} showed that the indicator functions of $(n-2)$-subcubes
(in $n$-dimensional discrete cube) maximize the $2$-stability with
$a=1/4$. When the mean is small, hypercontractivity inequalities
are a effective tool to address this case. In particular, Kahn, Kalai,
and Linial \cite{kahn1988influence} first applied the single-function
version of (forward) hypercontractivity inequalities to obtain bounds
for the $2$-stability problem, by substituting the nonnegative functions
in the hypercontractivity inequalities with the Boolean functions.
Mossel and O'Donnell \cite{mossel2006non,O'Donnell14analysisof} applied
the two-function version of hypercontractivity inequalities to obtain
bounds for the two-function version of $\alpha$-stability problem
in a similar way. Kahn, Kalai, and Linial's result as well as Mossel
and O'Donnell's are known as \emph{small-set expansion theorems}.
Kamath and Anantharam \cite{kamath2016non} slightly strengthened
the small-set expansion theorems via utilizing hypercontractivity
inequalities in a slightly different way. All these bounds derived
by hypercontractivity inequalities are asymptotically sharp in certain
sense when the mean approaches zero \cite{O'Donnell14analysisof}.
A variant of the $\alpha$-stability problem with the means vanishing
exponentially as $n\to\infty$ was studied in \cite{ordentlich2020note,kirshner2019moment,yu2021graphs,yu2021strong}.
Various stronger version of hypercontractivity inequalities were derived
or used to obtain sharper bounds on the $\alpha$-stability in these
papers. The optimal exponent for the $\alpha$-stability problem in
this variant setting was explicitly given in entropy optimization
forms in \cite{yu2021graphs,yu2021strong}, not only for the binary
random vectors $(\mathbf{X},\mathbf{Y})$, but also for the random
vectors defined on very general spaces (Polish spaces).

The maximal $\Phi$-stability problem in the Gaussian setting with
$\Phi$ restricted to be convex and increasing was fully resolved
by Borell \cite{borell1985geometric} in 1985. In particular, he showed
that the $\Phi$-stability is maximized by the indicators of half-spaces
over all measurable Boolean functions $f:\mathbb{R}^{n}\to\{0,1\}$
of the same measure; see an explicit statement of this result in \cite{kindler2015remarks}.
Such a result is known as \emph{Borell's Isoperimetric Theorem}. The
Gaussian analogues of the Courtade--Kumar conjecture and Li--Médard
conjecture were consequences of Borell's Isoperimetric Theorem, which
were also proved respectively by Kindler, O'Donnell, and Witmer \cite{kindler2015remarks}
and by Eldan \cite{eldan2015two} using alternative approaches.

\subsection{Organization}

This paper is organized as follows. In Sections 2 and 3, we present
our main results and the improvements in detail. The proofs of the
main results and related lemmas, propositions, and corollaries are
provided in Sections 4-9.

\section{Main Results}

In this paper, we aim at determining the maximal $\Phi$-stability
for given volume $a\in(0,1)$. By using Fourier analysis, we first
prove a general bound on the maximal $\Phi$-stability. 
\begin{thm}
\label{thm:generalbound} For $a,\rho\in(0,1)$, $\mathbf{MaxStab}_{\Phi}(a)\leq\Gamma_{\rho}(a)$,
where 
\begin{align}
\Gamma_{\rho}(a):=\sup_{P_{Z|S}} & \mathbb{E}[\Phi(a+\rho Z)]\label{eq:generalbound}\\
\mathrm{s.t.}\; & 0\leq a+\rho Z\leq1\quad\mathrm{a.s.}\label{eq:con1}\\
 & \mathbb{E}[Z]=0\label{eq:con2}\\
 & \mathbb{E}[Z^{2}]\leq\mathbb{E}[SZ]\label{eq:con3}
\end{align}
with\footnote{Throughout this paper, we use $P_{X}$ to denote the probability mass
function of a random variable $X$, and use $P_{Y|X}$ to denote the
conditional probability mass function of a random variable $Y$ given
$X$.} 
\begin{equation}
P_{S}(s)=\begin{cases}
1-a & s=-a\\
a & s=1-a
\end{cases}.\label{eq:P_S}
\end{equation}
\end{thm}

The supremum in \eqref{eq:generalbound} is taken over all conditional
probability mass functions (conditional pmf) $P_{Z|S}$, and the random
variables $(S,Z)$ in the objective function and constraints follow
the joint distribution $P_{SZ}:=P_{S}P_{Z|S}$.

Note that the constraint in \eqref{eq:con3} effectively dominates
(i.e., upper bounds) the energy of $Z$, since obviously, this constraint
implies $\mathbb{E}[Z^{2}]\leq\mathbb{E}[SZ]\leq\mathbb{E}[S^{2}]=a(1-a)$.
On the other hand, $\Phi$ is convex, and hence, the objective function
in \eqref{eq:generalbound} is dominated if the energy of $Z$ is
dominated. In other words, the constraint \eqref{eq:con3} dominates
the objective function via dominating the energy of $Z$. 
\begin{proof}[Proof of Theorem \ref{thm:generalbound}]
Our proof relies on Boolean Fourier analysis \cite{O'Donnell14analysisof}.
Consider the Fourier basis\footnote{Throughout this paper, we denote $[m:n]:=\{m,m+1,...,n\}$. When $m=1$,
for brevity, we denote $[n]:=[1:n]$. } $\{\chi_{\mathcal{S}}\}_{\mathcal{S}\subseteq[n]}$ with $\chi_{\mathcal{S}}(\mathbf{x}):=\prod_{i\in\mathcal{S}}x_{i}$
for $\mathcal{S}\subseteq[n]$. Then for a function $f:\{-1,1\}^{n}\to\mathbb{R}$,
define its Fourier coefficients as 
\begin{align}
 & \hat{f}_{\mathcal{S}}:=\mathbb{E}_{\mathbf{X}\sim\mathrm{Unif}\{-1,1\}^{n}}[f(\mathbf{X})\chi_{\mathcal{S}}(\mathbf{X})],\;\mathcal{S}\subseteq[n].\label{eq:Fourier}
\end{align}
Then the Fourier expansion of the function $f$ (cf. \cite[Equation (1.6)]{O'Donnell14analysisof})
is 
\[
f(\mathbf{x})=\sum_{\mathcal{S}\subseteq[n]}\hat{f}_{\mathcal{S}}\chi_{\mathcal{S}}(\mathbf{x}).
\]
The \emph{degree-$k$ Fourier weight }of $f$ is defined as 
\begin{align}
\mathbf{W}_{k}[f] & :=\sum_{\mathcal{S}:|\mathcal{S}|=k}\hat{f}_{\mathcal{S}}^{2},\quad k\in[0:n].\label{eq:FourierWeight}
\end{align}
For brevity, we denote $\mathbf{W}_{k}[f]$ as $\mathbf{W}_{k}$.

By the definition of Fourier weights, it is easily seen that for a
Boolean function with volume $\mathbb{E}f=a$, 
\[
\mathbf{W}_{0}=a^{2}\textrm{ and }\sum_{k=0}^{n}\mathbf{W}_{k}=a.
\]
Define the \emph{degree-$k$ part }of $f$ as 
\[
f_{k}(\mathbf{x}):=\sum_{\mathcal{S}\subseteq[n]:|\mathcal{S}|=k}\hat{f}_{\mathcal{S}}\chi_{\mathcal{S}}(\mathbf{x}).
\]
Then it is easy to check that 
\begin{align*}
 & T_{\rho}f(\mathbf{x})=\sum_{k=0}^{n}\rho^{k}f_{k}(\mathbf{x})\\
 & \sum_{k=0}^{n}f_{k}(\mathbf{x})=f(\mathbf{x})\in[0,1],\forall\mathbf{x}\\
 & f_{0}(\mathbf{x})=a,\forall\mathbf{x}\\
 & \sum_{k=0}^{n}\mathbb{E}[f_{k}(\mathbf{X})^{2}]=\sum_{k=0}^{n}\mathbf{W}_{k}=a\\
 & \mathbb{E}[f_{k}(\mathbf{X})f_{k'}(\mathbf{X})]=0,0\le k'<k\le n.
\end{align*}
Denote 
\begin{align}
Z & :=\sum_{k=1}^{n}\rho^{k-1}f_{k}(\mathbf{X})\label{eq:S}\\
S & :=f-a=\sum_{k=1}^{n}f_{k}(\mathbf{X}).\label{eq:Z}
\end{align}
Then 
\begin{align*}
 & T_{\rho}f(\mathbf{X})=a+\rho Z,\qquad\mathbb{E}Z=0,
\end{align*}
and $S$ follows the distribution in \eqref{eq:P_S}. Observe that
\begin{align*}
 & \mathbb{E}[Z^{2}]=\sum_{k=1}^{n}\rho^{2(k-1)}\mathbf{W}_{k}\leq\sum_{k=1}^{n}\rho^{k-1}\mathbf{W}_{k}=\mathbb{E}[SZ].
\end{align*}
Hence $\mathbf{MaxStab}_{\Phi}(a)\leq\Gamma_{\rho}(a)$. 
\end{proof}
We next provide several properties of the optimization problem in
\eqref{eq:generalbound}. Denote $\mathcal{S}:=\{-a,1-a\}$. The set
$A:=\{(\Phi(a+\rho z),z,z^{2}-sz):\,0\leq a+\rho z\le1\}$ lies in
the $3$-dimensional Euclidean space, and $\mathbb{E}[(\Phi(a+\rho Z),Z,Z^{2}-sZ)|S=s]$
lies in the convex hull of $A$. Since $A$ is connected, by Bunt's
extension of Carathéodory's theorem \cite{bunt1934bijdrage} or a
more general version \cite[Theorem 18]{eggleston1966convexity}, each
point in the convex hull of $A$ is a convex combination of at most
$3$ points in $A$. This further implies that without loss of optimality,
it suffices to restrict $P_{Z|S}$ in the optimization in \eqref{eq:generalbound}
such that $|\mathrm{supp}(P_{Z|S=s})|\le3$ for each $s\in\mathcal{S}$.
As a consequence, the supremum in \eqref{eq:generalbound} is actually
a maximum, since the feasible region is compact and the objective
and constraint functions are continuous.

For the optimization problem in \eqref{eq:generalbound}, we also
claim that without loss of optimality, it suffices to restrict $|\mathrm{supp}(P_{Z})|\le4$.
First, observe that the $\sup_{P_{Z|S}}$ in \eqref{eq:generalbound}
can be rewritten as three-fold optimizations: $\sup_{\mathcal{Z}}\sup_{P_{S|Z}}\sup_{P_{Z}}$,
where the first sup is taken over all countable subsets $\mathcal{Z}$
of $\mathbb{R}$ with each element $z$ satisfying $0\leq a+\rho z\leq1$,
the second sup is taken over all conditional pmfs $P_{S|Z}$ (specifically,
$\{P_{S|Z=z},z\in\mathcal{Z}\}$), and the last sup is taken over
all pmfs $P_{Z}$ on $\mathcal{Z}$ such that $\sum_{z\in\mathcal{Z}}P_{S|Z}(-a|z)P_{Z}(z)=P_{S}(-a)$
(the analogous equality for $s=1-a$ automatically holds if this equality
holds) and constraints \eqref{eq:con2} and \eqref{eq:con3} hold.
Here by convention, if there is no feasible solution for a supremization,
its value is set to $-\infty$. Given $\mathcal{Z}$ and $P_{S|Z}$,
the third sup above, $\sup_{P_{Z}}$, is in fact a linear program.
Following arguments similar to the one in the last paragraph, one
can restrict $|\mathrm{supp}(P_{Z})|\le4$.

Let $m\ge3$. For $s\in\mathcal{S}$, denote $(z_{s,i})_{i\in[m]}\in\mathbb{R}^{m}$
as a vector consisting of all (possibly repetitions of) the elements
in the support of $P_{Z|S=s}$. Denote $(p_{s,i})_{(s,i)\in\mathcal{S}\times[m]}$
as a joint pmf on $\mathcal{S}\times[m]$ whose marginal on $\mathcal{S}$
is $P_{S}$. Then, the optimization in \eqref{eq:generalbound} is
equivalently taken over all the tuples $(z_{s,i},p_{s,i})_{(s,i)\in\mathcal{S}\times[m]}$
such that $\sum_{i\in[m]}p_{s,i}=P_{S}(s),\forall s\in\mathcal{S}$.
Without loss of optimality, we may additionally assume $p_{s,i}>0,\forall s\in\mathcal{S},i\in[m]$,
since if $p_{s,i}=0,p_{s,j}>0$, then we can re-choose $p_{s,i}\leftarrow\frac{p_{s,j}}{2},p_{s,j}\leftarrow\frac{p_{s,j}}{2},z_{s,i}\leftarrow z_{s,j}$
which leads to the same distribution $P_{SZ}$. Then, the arguments
above lead us to define 
\begin{align}
\Lambda_{\rho}^{(m)}(a):=\max_{z_{s,i}\in\mathbb{R},p_{s,i}>0,(s,i)\in\mathcal{S}\times[m]} & \sum_{(s,i)\in\mathcal{S}\times[m]}p_{s,i}\Phi(a+\rho z_{s,i})\label{eq:generalboundsimple}\\
\mathrm{s.t.}\; & \sum_{i\in[m]}p_{s,i}=P_{S}(s),\forall s\in\mathcal{S},\label{eq:con1-1}\\
 & 0\leq a+\rho z_{s,i}\leq1,\forall(s,i)\in\mathcal{S}\times[m],\label{eq:con2-1}\\
 & \sum_{(s,i)\in\mathcal{S}\times[m]}p_{s,i}z_{s,i}=0,\label{eq:con3-1}\\
 & \sum_{(s,i)\in\mathcal{S}\times[m]}p_{s,i}(z_{s,i}^{2}-sz_{s,i})\leq0.\label{eq:con4-1}
\end{align}
As discussed above, without changing the value of the optimization,
we can add the additional constraint $|\{z_{s,i}\}_{(s,i)\in\mathcal{S}\times[m]}|\le4$
into the maximization problem in \eqref{eq:generalboundsimple}. Combining
all the points above, we arrive at the following results. 
\begin{prop}
\label{prop:property} For $a,\rho\in(0,1)$ and $m\ge3$, the following
hold. \\
 1. It holds that $\Gamma_{\rho}(a)=\Lambda_{\rho}^{(m)}(a)$. (Hence
we denote $\Lambda_{\rho}(a):=\Lambda_{\rho}^{(m)}(a)$ for $m\ge3$.)\\
 2. Any optimal solution to the optimization problem in \eqref{eq:generalbound}
(or the equivalent one in \eqref{eq:generalboundsimple}) satisfies
$\mathbb{E}[Z^{2}]=\mathbb{E}[SZ]$.\\
 3. Any optimal solution to the maximization problem in \eqref{eq:generalboundsimple}
satisfies that $z_{1-a,j}\ge z_{-a,i}+\frac{1}{2}$ for all $i,j\in[m]$
such that $z_{-a,i}>\frac{-a}{\rho},\,z_{1-a,j}<\frac{1-a}{\rho}$.
Moreover, this condition is satisfied by some $i,j\in[m]$.\\
 4. The linear independence constraint qualification (LICQ) \footnote{The LICQ for a maximizer point $\mathbf{x}^{*}$ is the condition
that the gradients of the active inequality constraints and the gradients
of the equality constraints are linearly independent at $\mathbf{x}^{*}$.
We refer readers to \cite{peterson1973review,bazaraa2013nonlinear}
for more details on LICQ.} is satisfied for the maximization problem in \eqref{eq:generalboundsimple}. 
\end{prop}

Statement 1 of Proposition \ref{prop:property} has already been derived
above, and Statements 2-4 are proven in Section \ref{sec:Proof-of-Proposition}.

Proposition \ref{prop:property} enables us to employ the Karush--Kuhn--Tucker
(KKT) conditions to simplify the bound in Theorem \ref{thm:generalbound}.
In fact, the bound in Theorem \ref{thm:generalbound} (or the equivalent
one in \eqref{eq:generalboundsimple}) is general enough to apply
to several common cases, e.g., the asymmetric and symmetric $\alpha$-stabilities,
as follows.

\subsection{Symmetric $\Phi$}

We first consider symmetric $\Phi$. Here $\Phi$ is said to be symmetric
(w.r.t. $\frac{1}{2}$) if $\Phi(\frac{1}{2}-t)=\Phi(\frac{1}{2}+t)$
for $t\in[0,\frac{1}{2}]$. Define 
\begin{equation}
\overline{\Gamma}_{\rho}(a):=\sup_{(z_{1},z_{2})\in\overline{\mathcal{Z}}}(1-a-p)\Phi(0)+p\Phi(a+\rho z_{1})+q\Phi(a+\rho z_{2})+(a-q)\Phi(1),\label{eq:-29}
\end{equation}
where 
\begin{align}
 & p=\frac{(1-a)a(1-\rho)}{(1-\rho-\rho z_{1}+\rho z_{2})(a+\rho z_{1})},\label{eq:-7}\\
 & q=\frac{(1-a)a(1-\rho)}{(1-\rho-\rho z_{1}+\rho z_{2})(1-a-\rho z_{2})},\label{eq:-8}
\end{align}
and 
\begin{align*}
\overline{\mathcal{Z}}:= & \biggl\{(z_{1},z_{2})\in(\frac{-a}{\rho},\frac{1-a}{\rho})^{2}:z_{2}-z_{1}\geq\frac{1}{2},\\
 & z_{2}-z_{1}\geq\frac{(1-\rho)z_{2}}{1-a-\rho z_{2}},\;z_{2}-z_{1}\geq-\frac{(1-\rho)z_{1}}{a+\rho z_{1}}\biggr\}.
\end{align*}
Define 
\begin{equation}
\hat{\Gamma}_{\rho}(a):=\sup_{\max\{\frac{-a}{\rho},c\}\le\tilde{z}\le\frac{-a(1+b)}{2(1-a)}}(1-a-p)\Phi(a+\rho\tilde{z})+p\Phi(a+\rho(b-\tilde{z}))+a\Phi(a+\rho\hat{z}),\label{eq:-37}
\end{equation}
where $b=\frac{1-2a}{\rho}$, $c=\frac{1}{2}(b-\sqrt{\frac{a+2ab+b^{2}}{1-a}})$,
and 
\begin{align}
 & \hat{z}=\frac{1}{2}(\sqrt{\frac{\Delta}{a}}+b+1),\label{eq:-25-1}\\
 & p=\frac{-\sqrt{a\Delta}-(a(b+1-2\tilde{z})+2\tilde{z})}{2(b-2\tilde{z})}\label{eq:-26-1}
\end{align}
with $\Delta=a(b^{2}-4b\tilde{z}+2b+4\tilde{z}^{2}+1)+4\tilde{z}(b-\tilde{z})$.
For symmetric $\Phi$, we prove the following result, whose proof
is provided in Section \ref{sec:Proof-of-Theorem-symmetric}. 
\begin{thm}
\label{thm:symmetric} Assume that $\Phi$ is symmetric and continuous
on $[0,1]$, and differentiable on $(0,1)$ whose derivative $\Phi'$
is increasing and continuous. Assume $\rho\in(0,1)$. \\
 1) If additionally, $\Phi'$ is strictly concave on $(0,\frac{1}{2}]$,
then for $a\in(0,\frac{1}{2}]$, $\mathbf{MaxStab}_{\Phi}(a)\leq\Lambda_{\rho}(a)=\overline{\Gamma}_{\rho}(a)$.
Moreover, for $a=\frac{1}{2}$, we can additionally assume $z_{1}+z_{2}=0$
in the supremization in \eqref{eq:geq}. \\
 2) If additionally, $\Phi'$ is strictly convex on $(0,\frac{1}{2}]$,
then for $a\in(0,\frac{1}{2}]$, $\mathbf{MaxStab}_{\Phi}(a)\leq\Lambda_{\rho}(a)=\max\{\hat{\Gamma}_{\rho}(a),\hat{\Gamma}_{\rho}(1-a),\mathbb{E}[\Phi(a+\rho S)]\}$.
In particular, for $a=\frac{1}{2}$, we have $\mathbf{MaxStab}_{\Phi}(\frac{1}{2})\leq\Lambda_{\rho}(\frac{1}{2})=\Phi(\frac{1-\rho}{2}).$ 
\end{thm}

Note that $\Phi_{\alpha}^{\mathrm{sym}}$ with $\alpha\in(1,2)\cup(3,\infty)$
satisfies the assumption in Statement 1 of Theorem \ref{thm:symmetric},
and $\Phi_{\alpha}^{\mathrm{sym}}$ with $\alpha\in(2,3)$ satisfies
the assumption in Statement 2 of Theorem \ref{thm:symmetric}. We
now apply Theorem \ref{thm:symmetric} to the symmetric $\alpha$-stability.
Let $\rho^{*}$ be the solution in $(0,1)$ to the equation 
\begin{equation}
\psi(\rho):=(1+\rho^{2})\log(\frac{1+\rho}{2})-(1-\rho)^{2}\log(\frac{1-\rho}{2})=0.\label{eq:psi}
\end{equation}
We have that $\rho^{*}\approx0.461491$. Then we prove the Courtade--Kumar
conjecture for $\rho\le\rho^{*}$. The proof is provided in Section
\ref{sec:Proof-of-Theorem-CKConjecture-1}. 
\begin{cor}
\label{cor:CK_Conjecture} $\mathbf{MaxStab}_{1}^{\mathrm{sym}}(\frac{1}{2})$
is attained by dictator functions for $0\le\rho\le\rho^{*}$. 
\end{cor}

Samorodnitsky \cite{samorodnitsky2016entropy} proved the Courtade--Kumar
conjecture for $\rho\in[0,\rho_{0}]$, where $0<\rho_{0}<1$ is some
absolute constant. However, the value of $\rho_{0}$ was not explicitly
given in \cite{samorodnitsky2016entropy}, and also it was assumed
to be ``sufficiently small''. In contrast, the value of $\rho^{*}$
in Corollary \ref{cor:CK_Conjecture} is explicitly given. As for
the proof ideas, Samorodnitsky's proof in \cite{samorodnitsky2016entropy}
and our proof in this paper are both based on Fourier analysis (and
also quantitative versions of Friedgut--Kalai--Naor theorems which
are used to further improve the threshold in Section \ref{subsec:Further-Improvement}),
but use it in different ways. For example, different from the linear
programs in \cite{samorodnitsky2016entropy}, the duality for \emph{nonconvex}
programs is used in this paper.

It is worth noting that in \cite{samorodnitsky2016entropy}, the mean
$a$ is not fixed. For this case, we numerically evaluate $\sup_{a\in[0,1/2]}\overline{\Gamma}_{\rho}(a)-\Phi_{1}^{\mathrm{sym}}(a)$,
and observe that this supremum is attained at $a=1/2$ for $0\le\rho\le\rho^{*}$.
This indicates that Statement 1 of Theorem \ref{thm:symmetric} seems
to imply the original Courtade--Kumar conjecture for $0\le\rho\le\rho^{*}$,
i.e., the conjecture that dictator functions maximize the mutual information
$I(f(\mathbf{Y});\mathbf{X})$ over all Boolean functions $f$. It
remains to find a proof for this observation in the future.

We also generalize the corollary above to the case $\alpha\in[1,5]$.
The proof of the following corollary is provided in Section \ref{sec:Proof-of-Theorem-symmetricpower}. 
\begin{cor}
\label{cor:symmetricpower}For $(\rho,\alpha)$ such that $0\le\rho\leq\begin{cases}
\rho^{*} & \alpha\in[1,2)\\
1 & \alpha\in[2,5]
\end{cases},$ $\mathbf{MaxStab}_{\alpha}^{\mathrm{sym}}(\frac{1}{2})$ is attained
by dictator functions. 
\end{cor}

Corollary \ref{cor:symmetricpower} resolves the Mossel--O'Donnell
conjecture for $\alpha\in[2,5]$. However, our result for cases $\alpha=2,3$
is not new. The case $\alpha=2$ was resolved by Witsenhausen \cite{witsenhausen1975sequences}
by using the tensorization property of maximal correlation, and the
case $\alpha=3$ was resolved by Mossel and O'Donnell \cite{mossel2005coin}
by reducing the case $\alpha=3$ to the case $\alpha=2$. Note that
Mossel and O'Donnell's method seems difficult to extend to the case
$\alpha>3$. Moreover, the case $\alpha\in[2,3]$ can be also obtained
by combining Mossel and O'Donnell's result and Lemma \ref{lem:BarnesOzgur}.
Our Corollary \ref{cor:symmetricpower}, proven by using Fourier analysis,
is a generalization of Witsenhausen's and Mossel--O'Donnell's results.
Moreover, combining Corollary \ref{cor:symmetricpower} and the counterexample
example for $\alpha=10$ in \cite{mossel2005coin} yields that the
estimation of $\breve{\alpha}_{\max}$ is improved to $5\le\breve{\alpha}_{\max}\leq10$.

\subsection{Asymmetric $\Phi$}

We next consider asymmetric $\Phi$. Define\footnote{Indeed, the range of $z_{2}$ can be further restricted to $\frac{1-a}{2-\rho}\le z_{2}\le1-a$.
For simplicity, we only restrict $0\le z_{2}\le1-a$, which is sufficient
to show our result on the asymmetric $\alpha$-stability in Corollary
\ref{cor:asymmetricpower}. } 
\[
\widetilde{\Gamma}_{\rho}(a)=\sup_{0\le z_{2}\le1-a}(1-a-p)\Phi(0)+p\Phi(a+\rho z_{1})+a\Phi(a+\rho z_{2}),
\]
where 
\begin{align}
z_{1} & =\frac{z_{2}(\rho(1-z_{2})-a)}{1-a-\rho z_{2}},\label{eq:z1}\\
p & =\frac{a(1-a-\rho z_{2})^{2}}{a+\rho^{2}z_{2}-(a+\rho z_{2})^{2}}.\label{eq:p}
\end{align}
Then following proof steps similar to those of Theorem \ref{thm:symmetric},
we have the following result. The proof is omitted. 
\begin{thm}
\label{thm:derivativeconcave} If $\Phi$ is continuous on $[0,1]$
and differentiable on $(0,1)$ whose derivative $\Phi'$ is increasing,
continuous, and strictly concave on $(0,1)$, then for $a,\rho\in(0,1)$,
\begin{align}
 & \mathbf{MaxStab}_{\Phi}(a)\leq\Lambda_{\rho}(a)=\widetilde{\Gamma}_{\rho}(a).\label{eq:-9}
\end{align}
\end{thm}

\begin{rem}
\label{rem:convex}If $\Phi'$ is increasing, continuous, and strictly
convex, then by redefining $\tilde{\Phi}:t\in[0,1]\mapsto\Phi(1-t)$
and substituting $a\leftarrow1-a$, \eqref{eq:-9} still holds. 
\end{rem}

We now consider the asymmetric $\alpha$-stability. Denote $\theta(\alpha)$
as the solution to 
\begin{equation}
\theta^{2-\alpha}+\frac{1}{\alpha}(1-\theta)=1\label{eq:theta-1}
\end{equation}
with unknown $\theta$ for given $\alpha\in(1,2)$. Then the following
holds, whose proof is provided in Section \ref{sec:Proof-of-Theorem-asymmetricpower}. 
\begin{cor}
\label{cor:asymmetricpower}For $(\rho,\alpha)$ such that $0\le\rho\leq\begin{cases}
\frac{1-\theta(\alpha)}{1+\theta(\alpha)} & \alpha\in(1,2)\\
1 & \alpha\in[2,3]
\end{cases},$ $\mathbf{MaxStab}_{\alpha}(\frac{1}{2})$ is attained by dictator
functions. 
\end{cor}

This result partially resolves the Li--Médard conjecture for the
case of $0\le\rho\leq\frac{1-\theta(\alpha)}{1+\theta(\alpha)},\alpha\in(1,2)$.
This region is plotted in Fig. \ref{fig:The-region-of}. The question
whether $\mathbf{MaxStab}_{\alpha}(\frac{1}{2})$ is attained by dictator
functions for the case $\alpha=3$ was posed in \cite{li2019boolean}.
Our Corollary \ref{cor:asymmetricpower} gives a positive answer to
this question. Combining Corollary \ref{cor:asymmetricpower} and
the counterexample for $\alpha=10$ in \cite{mossel2005coin} implies
$3\le\alpha_{\max}\leq10$.

\begin{figure}
\centering \includegraphics[scale=0.7]{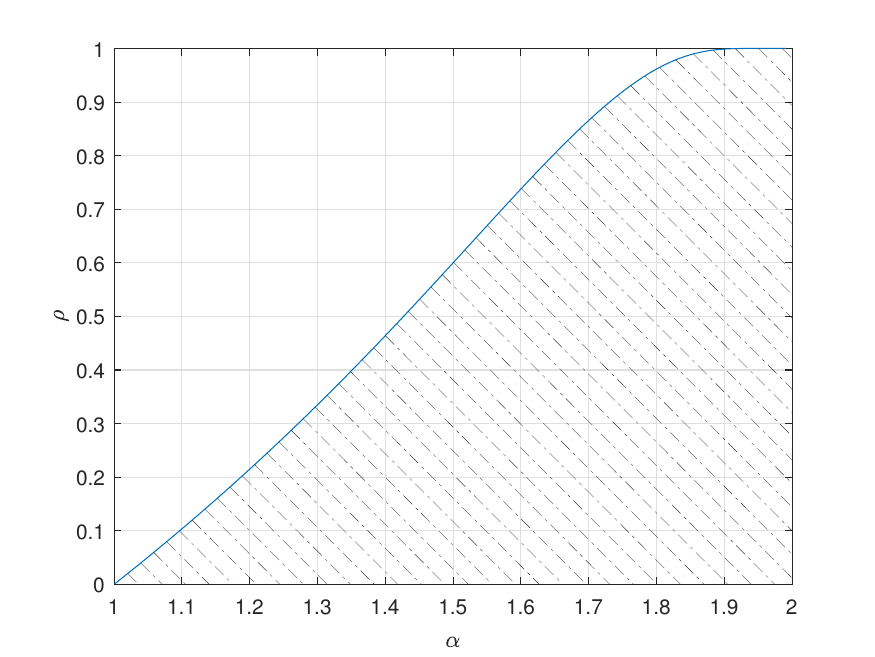}

\caption{\label{fig:The-region-of}The region of $0\le\rho\protect\leq\frac{1-\theta(\alpha)}{1+\theta(\alpha)}$.}
\end{figure}

\section{Further Improvement of Our Results}

In this section, we further improve our bounds on the maximal $\Phi$-stability
by applying an improved version of the Friedgut--Kalai--Naor (FKN)
theorem \cite{friedgut2002boolean}. Hence, to this end, we improve
the FKN theorem first.

\subsection{\label{sec:Improvements}Improvements of FKN Theorem}

The FKN theorem concerns about which Boolean functions $f$ on the
discrete cube have Fourier coefficients concentrated at the lowest
two levels. It states that such Boolean functions are close to either
a constant function (i.e., $f=0$ or $1$) or a dictator function
($f=1\{x_{i}=1\}$ or $1\{x_{i}=-1\}$). Here we aim at improving
the FKN theorem by focusing on the class of Boolean functions with
a given mean. For this case, the FKN theorem can be also formulated
as a theorem about maximizing the first-order Fourier weight of a
Boolean function given the maximum of its correlations to all dictator
functions. We next provide the formulation.

For $a,\beta\in2^{-n}[0:2^{n}]$, define

\begin{equation}
W^{(n)}(a,\beta):=\max_{f:\{-1,1\}^{n}\to\{0,1\}:\mathbb{E}f=a,\max_{i\in[n]}|\hat{f}_{\{i\}}|=\beta}\mathbf{W}_{1}[f].\label{eq:W}
\end{equation}
The quantity $W^{(n)}(a,\beta)$ was implicitly studied by Friedgut,
Kalai, and Naor \cite{friedgut2002boolean} who showed that for $a=\frac{1}{2}$,
$W(\frac{1}{2},\beta)\to\frac{1}{4}$ if and only if $\beta\to\frac{1}{2}$.

Similarly, we define 
\[
W^{(n)}(a):=\max_{f:\{-1,1\}^{n}\to\{0,1\}:\mathbb{E}f=a}\mathbf{W}_{1}[f].
\]
Here without ambiguity, we use the same notation $W^{(n)}$ but with
different numbers of parameters to denote two different functions
$W^{(n)}(a,\beta)$ and $W^{(n)}(a)$. We first use $W^{(n)}(a)$
to bound $W^{(n)}(a,\beta)$ in the following lemma. The proof of
Lemma \ref{lem:W-W} is provided in Section \ref{sec:Proof-of-Lemma}. 
\begin{lem}
\label{lem:W-W} For any Boolean function $f:\{-1,1\}^{n}\to\{0,1\}$,
we have 
\begin{align}
\mathbf{W}_{1}[f] & \leq\beta^{2}+\left(\sqrt{W^{(n)}(a)-a^{2}}+\sqrt{W^{(n-1)}(a-\beta)}\right)^{2}\label{eq:-58}
\end{align}
where $a:=\mathbb{E}f$ and $\beta:=|\hat{f}_{\{1\}}|$. In particular,
for balanced Boolean functions (i.e., $a=\frac{1}{2}$), we have 
\begin{equation}
\mathbf{W}_{1}[f]\le\beta^{2}+W^{(n-1)}(\frac{1}{2}-\beta).\label{eq:-20-4}
\end{equation}
Moreover, for any possible values of $\beta$ such that there exists
at least one Boolean function with $\hat{f}_{\{1\}}=\beta$, equality
in \eqref{eq:-20-4} is attained by some balanced Boolean function. 
\end{lem}

\begin{rem}
\label{rem:-implies-that} \eqref{eq:-58} implies that $W^{(n)}(a,\beta)$
is no larger than the RHS of \eqref{eq:-58}, but \eqref{eq:-58}
is indeed stronger than this conclusion in the sense that given $\beta$,
\eqref{eq:-58} holds for any Boolean functions $f$ such that there
is at least one $i$ satisfying $|\hat{f}_{\{i\}}|=\beta$, without
requiring other $j\in[n]\backslash\{i\}$ satisfying $|\hat{f}_{\{j\}}|\le\beta$. 
\end{rem}

In fact, the quantity $W^{(n)}(a)$ was relatively well studied in
the literature. Define 
\[
\phi(t):=\begin{cases}
2t^{2}\ln\frac{1}{t} & 0<t\leq\frac{1}{4}\\
2t^{2}(\frac{1}{\sqrt{t}}-1) & 0<t\leq\frac{1}{4}\\
\frac{t}{2} & \frac{1}{4}\leq t\leq\frac{1}{2}
\end{cases}.
\]
Define for $0\leq t\leq1$, 
\[
\varphi(t):=\phi(\min\{t,1-t\}).
\]
It is known that \cite{O'Donnell14analysisof,fu2001minimum,yu2019improved}
for any $f:\{-1,1\}^{n}\to\{0,1\}$ with $\mathbb{E}f=a\in[0,1]$,
\begin{equation}
W^{(n)}(a)\leq\varphi(a).\label{eq:W-phi}
\end{equation}
The bounds $\frac{t}{2}$ and $2t^{2}(\frac{1}{\sqrt{t}}-1)$ were
respectively proven by Fu, Wei, and Yeung \cite{fu2001minimum} and
the present author and Tan \cite{yu2019improved}, by using linear
programming methods (together with MacWilliams--Delsarte identities).
The bound $2t^{2}\ln\frac{1}{t}$ is called \emph{Chang's bound} \cite{chang2002polynomial},
which can be proven by several methods, e.g., by hypercontractivity
inequalities \cite{O'Donnell14analysisof}.

Combining Lemma \ref{lem:W-W} and \eqref{eq:W-phi} yields the following
result. 
\begin{prop}
\label{prop:ImprovedFKN}For $0\le\beta\le a\le\frac{1}{2}$, 
\begin{align}
W^{(n)}(a,\beta) & \leq\beta^{2}+\big(\sqrt{\varphi(a)-a^{2}}+\sqrt{\varphi(a-\beta)}\big)^{2}.\label{eq:-71}
\end{align}
In particular, for $0\le\beta\le a=\frac{1}{2}$, 
\begin{equation}
W^{(n)}(\frac{1}{2},\beta)\leq\beta^{2}+\varphi(\frac{1}{2}-\beta).\label{eq:-20-4-1}
\end{equation}
\end{prop}

This bound improves the existing bound proven in \cite{jendrej2015some}.
Moreover, our bound is sharp for $a=\frac{1}{2},\beta=\frac{1}{4}$,
and asymptotically sharp for $a=\frac{1}{2},\beta\uparrow\frac{1}{2}$.
The sharpness for $a=\frac{1}{2},\beta=\frac{1}{4}$ can be seen from
that for $\beta=\frac{1}{4}$, $W^{(n-1)}(\frac{1}{4})=\varphi(\frac{1}{4})=\frac{1}{8}$
for $n\ge3$, which is attained by $1\{(x_{2},..,x_{n}):x_{2}=x_{3}=1\}$.
The asymptotic sharpness for $a=\frac{1}{2},\beta\uparrow\frac{1}{2}$
can be seen from that if we define $\mathcal{A}_{n}=\{(x_{2},..,x_{n}):\frac{1}{\sqrt{n-1}}\sum_{i=2}^{n}x_{i}\geq r_{n}\}$
with $\lim_{n\to\infty}r_{n}=Q^{-1}(\varepsilon)$ as a sequence of
Hamming balls with volumes asymptotically approaching $\varepsilon$,
where $Q$ is the Q-function, then we have that for $\beta=\frac{1}{2}-\varepsilon$,
$\lim_{n\to\infty}\mathbf{W}_{1}[1_{\mathcal{A}_{n}}]=\psi(Q^{-1}(\varepsilon))\sim2\varepsilon^{2}\ln\frac{1}{\varepsilon}$
as $\varepsilon\downarrow0$ \cite{O'Donnell14analysisof}, where
$\psi$ is the probability density function of the standard Gaussian.

As mentioned in Remark \ref{rem:-implies-that}, in Lemma \ref{lem:W-W}
we have not used the information $|\hat{f}_{\{i\}}|\le\beta$ for
all $i\in[2:n]$. Hence, Lemma \ref{lem:W-W} (or Proposition \ref{prop:ImprovedFKN})
cannot provide a good bound when $\beta$ is small, since for this
case $|\hat{f}_{\{i\}}|$ could be significantly larger than $\beta$
for some $i\in[2:n]$. In order to obtain a good bound for this case,
we use another method to derive the following quantitative (non-asymptotic)
version of FKN theorem, which is a consequence of the variant of Khintchine's
inequality proven by König, Schütt, and Tomczak-Jaegermann \cite{konig1999projection}.
It is worth noting that the variant of Khintchine's inequality proven
in \cite{konig1999projection} was also applied by Friedgut, Kalai,
and Naor to prove the original (asymptotic) version of the FKN theorem
\cite{friedgut2002boolean}. 
\begin{prop}
For $0\le\beta\le a=\frac{1}{2}$, 
\begin{equation}
W^{(n)}(\frac{1}{2},\beta)\leq\frac{1}{4}\Big(\sqrt{4(\frac{1}{2}-\frac{1}{\sqrt{2\pi}})\beta+\frac{1}{2\pi}}+\frac{1}{\sqrt{2\pi}}\Big)^{2}.\label{eq:-21}
\end{equation}
\end{prop}

\begin{proof}
We have that 
\begin{align}
\mathbf{W}_{1}[f] & =\mathbb{E}\Big[\sum_{i=1}^{n}\hat{f}_{\{i\}}X_{i}f(\mathbf{X})\Big]\label{eq:-45}\\
 & \leq\mathbb{E}\Big[\sum_{i=1}^{n}\hat{f}_{\{i\}}X_{i}\,1\{\sum_{i=1}^{n}\hat{f}_{\{i\}}X_{i}\geq0\}\Big]\label{eq:-43}\\
 & =\frac{1}{2}\mathbb{E}\Big[|\sum_{i=1}^{n}\hat{f}_{\{i\}}X_{i}|\Big]\nonumber \\
 & \leq\frac{1}{2}\Big(\sqrt{\frac{2}{\pi}}\sqrt{\mathbf{W}_{1}[f]}+(1-\sqrt{\frac{2}{\pi}})\beta\Big),\label{eq:-31}
\end{align}
where \eqref{eq:-43} follows since if we relax $(\hat{f}_{\{i\}})_{i\in[n]}$
and $f$ to be independent quantities, then given $\hat{f}_{\{i\}}$,
the Boolean function $f:\mathbf{x}\mapsto1\{\sum_{i=1}^{n}\hat{f}_{\{i\}}x_{i}\geq0\}$
maximizes the expectation in \eqref{eq:-45}, and \eqref{eq:-31}
follows from the following variant of Khintchine's inequality 
\[
\Big|\mathbb{E}\Big[\Big|\sum_{i=1}^{n}c_{i}X_{i}\Big|\Big]-\sqrt{\frac{2}{\pi}}\left\Vert \mathbf{c}\right\Vert _{2}\Big|\leq\Big(1-\sqrt{\frac{2}{\pi}}\Big)\left\Vert \mathbf{c}\right\Vert _{\infty}
\]
with $\mathbf{c}:=(c_{1},c_{2},...,c_{n})$ which was proven by König,
Schütt, and Tomczak-Jaegermann \cite{konig1999projection}. Solving
the inequality in \eqref{eq:-31}, we obtain \eqref{eq:-21}. 
\end{proof}

\subsection{\label{subsec:Further-Improvement}Further Improvement of Our Bounds
on Noise Stability}

Based on upper bounds on $W^{(n)}(a,\beta)$ (e.g., \eqref{eq:-71}
and \eqref{eq:-21}), we improve Theorem \ref{thm:generalbound} as
follows. 
\begin{thm}
\label{thm:generalbound-k} Let $\omega(a,\beta)$ be an upper bound
on $W^{(n)}(a,\beta)$. Then for $a,\rho\in(0,1)$, $\mathbf{MaxStab}_{\Phi}(a)\leq\Upsilon_{\rho}(a)$,
where 
\begin{align}
\Upsilon_{\rho}(a):=\sup_{\beta,P_{Z|SX}} & \mathbb{E}[\Phi(a+\rho Z)]\label{eq:generalbound-2}\\
\mathrm{s.t.}\; & 0\leq\beta\leq a,1-a,\label{eq:con1-2}\\
 & 0\leq a+\rho Z\leq1\quad\mathrm{a.s.}\label{eq:con2-2}\\
 & \mathbb{E}Z=0,\label{eq:con3-2}\\
 & \mathbb{E}[XZ]=\beta,\label{eq:con4-2}\\
 & \mathbb{E}[Z^{2}]\leq(1-\rho)\omega(a,\beta)+\rho\mathbb{E}[SZ],\label{eq:con6-2}
\end{align}
with 
\begin{equation}
P_{SX}(s,x)=\begin{cases}
\frac{1-a+\beta}{2} & (s,x)=(-a,-1)\\
\frac{1-a-\beta}{2} & (s,x)=(-a,1)\\
\frac{a-\beta}{2} & (s,x)=(1-a,-1)\\
\frac{a+\beta}{2} & (s,x)=(1-a,1)
\end{cases}.\label{eq:P_SX}
\end{equation}
\end{thm}

\begin{proof}
Recall that $\{\hat{f}_{\mathcal{S}}\}$ are the Fourier coefficients
of $f$. WLOG, we assume $|\hat{f}_{\{1\}}|\geq|\hat{f}_{\{i\}}|$
for all $2\le i\le n$. If $\hat{f}_{\{1\}}\ge0$, denote $\beta:=\hat{f}_{\{1\}}$
and $X:=X_{1}$; otherwise, $\beta:=-\hat{f}_{\{1\}}$ and $X:=-X_{1}$.
Hence $\beta\ge0$. Denote 
\begin{align*}
Y & =\sum_{i=2}^{n}\hat{f}_{\{i\}}X_{i},\quad\hat{S}=\sum_{k=2}^{n}f_{k}(\mathbf{X}),\quad\hat{Z}=\sum_{k=2}^{n}\rho^{k-2}f_{k}(\mathbf{X}).
\end{align*}
Recall the definitions of $S,Z$ in \eqref{eq:S} and \eqref{eq:Z}.
Then, 
\begin{align}
 & S=f-a=\beta X+Y+\hat{S}\nonumber \\
 & Z=\beta X+Y+\rho\hat{Z}\nonumber \\
 & \mathbb{E}[XZ]=\beta\nonumber \\
 & \mathbb{E}[SZ]=\mathbf{W}_{1}+\rho\mathbb{E}[\hat{S}\hat{Z}],\label{eq:-4-1}
\end{align}
and $(S,X)$ follows the distribution in \eqref{eq:P_SX} (since $\mathbb{E}[SX]=\beta$).
Observe that 
\begin{align}
 & \mathbb{E}[\hat{Z}^{2}]=\sum_{k=2}^{n}\rho^{2(k-2)}\mathbf{W}_{k}\leq\sum_{k=2}^{n}\rho^{k-2}\mathbf{W}_{k}=\mathbb{E}[\hat{S}\hat{Z}].\label{eq:-5-2}
\end{align}
Combining \eqref{eq:-4-1}, \eqref{eq:-5-2}, and the fact that $\mathbf{W}_{1}\le\omega(a,\beta)$,
yields that 
\begin{align*}
\mathbb{E}[Z^{2}] & =\mathbf{W}_{1}+\rho^{2}\mathbb{E}[\hat{Z}^{2}]\\
 & \leq(1-\rho)\mathbf{W}_{1}+\rho(\mathbf{W}_{1}+\rho\mathbb{E}[\hat{S}\hat{Z}])\\
 & \leq(1-\rho)\omega(a,\beta)+\rho\mathbb{E}[SZ].
\end{align*}
Hence $\mathbf{MaxStab}_{\Phi}(a)\leq\Upsilon_{\rho}(a)$. 
\end{proof}
Here $\omega(a,\beta)$ can be chosen as the upper bound in \eqref{eq:-71}
for any $a$ or as the minimum of the upper bounds in \eqref{eq:-71}
and \eqref{eq:-21} for $a=1/2$. Compared to Theorem \ref{thm:generalbound},
the bound in Theorem \ref{thm:generalbound-k} is based on a more
elaborate analysis, which introduces the maximum of absolute values
of first-order Fourier coefficients as a parameter, i.e., $\beta$,
and then optimizes the bound over all possible $\beta$. To better
understand the intuition of Theorem \ref{thm:generalbound-k}, we
focus on the case $a=1/2$. For this case, if $\beta$ is close to
$1/2$, then \eqref{eq:con4-2} forces $Z$ to be almost linearly
dependent of the Boolean r.v. $X$ (note that the energy of $Z$ is
controlled under the constraint and \eqref{eq:con6-2}). This in turn
forces optimal solutions to \eqref{eq:generalbound-2} close to dictator
functions (for which $f(\mathbf{X})=\frac{1+X}{2}$, and $Y=\hat{S}=\hat{Z}=0,S=Z=X/2$
where $X=X_{i}$ or $-X_{i}$ for some $i\in[n]$). On the other hand,
if $\beta$ is far from $1/2$, then $\omega(a,\beta)$ becomes relatively
small (see \eqref{eq:-71} and \eqref{eq:-21}), which, combined with
\eqref{eq:con6-2}, forces the energy of $Z$ to be small. This in
turn decreases the objective function, and hence, this kind of $\beta$
is excluded from being optimal to the supremization in \eqref{eq:generalbound-2}.

Denote $\mathcal{X}:=\{\pm1\}$. For simplicity, we next focus on
the case $a=1/2$, and suppose that $\omega(\beta):=\omega(1/2,\beta)$
is continuous in $\beta$ and satisfies the trivial inequality $\omega(\beta)\le\frac{1}{4}$
for all $\beta$ (since $\mathbf{W}_{0}+\mathbf{W}_{1}\le\mathbb{E}[f^{2}]=a$
and $\mathbf{W}_{0}=a^{2}$). Similarly to the definition of $\Lambda_{\rho}^{(m)}(a)$
in \eqref{eq:generalboundsimple}, for such $\omega(\beta)$ and $m\ge4$,
we define 
\begin{align}
\Psi_{\rho}^{(m)}:=\max_{\substack{\beta,z_{s,x,i}\in\mathbb{R},\\
p_{s,x,i}>0,\\
(s,x,i)\in\mathcal{S}\times\mathcal{X}\times[m]
}
} & \sum_{(s,x,i)\in\mathcal{S}\times\mathcal{X}\times[m]}p_{s,x,i}\Phi(\frac{1}{2}+\rho z_{s,x,i})\label{eq:generalboundsimple-2}\\
\mathrm{s.t.}\; & 0\leq\beta\leq\frac{1}{2},\label{eq:con1-5-2}\\
 & 0\leq\frac{1}{2}+\rho z_{s,x,i}\leq1,\forall(s,x,i)\in\mathcal{S}\times\mathcal{X}\times[m],\\
 & \sum_{i\in[m]}p_{s,x,i}=P_{SX}(s,x),\forall(s,x)\in\mathcal{S}\times\mathcal{X},\label{eq:con2-5-2}\\
 & \sum_{(s,x,i)\in\mathcal{S}\times\mathcal{X}\times[m]}p_{s,x,i}z_{s,x,i}=0,\label{eq:-48}\\
 & \sum_{(s,x,i)\in\mathcal{S}\times\mathcal{X}\times[m]}p_{s,x,i}xz_{s,x,i}=\beta,\\
 & \sum_{(s,x,i)\in\mathcal{S}\times\mathcal{X}\times[m]}p_{s,x,i}(z_{s,x,i}^{2}-\rho sz_{s,x,i})\leq(1-\rho)\omega(\beta).\label{eq:con3-3}
\end{align}

Following proof steps similar to those of Proposition \ref{prop:property},
one can show that the optimization problem in \eqref{eq:generalbound-2}
satisfies the following properties. The proof is provided in Section
\ref{sec:Proof-of-Proposition-1}. 
\begin{prop}
\label{prop:property-1} Suppose that\footnote{This guarantees the existence of the optimal solutions to the optimization
problem in \eqref{eq:generalbound-2} or the equivalent one in \eqref{eq:generalboundsimple-2}. } $\Phi$ is symmetric and continuous on $[0,1]$, and differentiable
on $(0,1)$ whose derivative $\Phi'$ is increasing and continuous
$(0,1)$, and strictly concave on $(0,\frac{1}{2}]$. Suppose that
$\omega(\beta)$ is continuous in $\beta$ and satisfies that $\omega(\beta)\le\frac{1}{4}$
for all $\beta$. For $a=1/2,\rho\in(0,1)$, and $m\ge4$, the following
hold. \\
 1. Without loss of optimality, it suffices to restrict that $|\mathrm{supp}(P_{Z})|\le7$
and $|\mathrm{supp}(P_{Z|S=s,X=x})|\le4$ for $(s,x)\in\mathcal{S}\times\mathcal{X}$,
which implies that $\Upsilon_{\rho}(\frac{1}{2})=\Psi_{\rho}^{(m)}$.
\\
 2. Any optimal solution to the optimization problem in \eqref{eq:generalbound-2}
with $a=1/2$ (or the equivalent one in \eqref{eq:generalboundsimple-2})
satisfies $\mathbb{E}[Z^{2}]=(1-\rho)\omega(\beta)+\rho\mathbb{E}[SZ]$.\\
 3. Any optimal solution to the maximization problem in \eqref{eq:generalboundsimple-2}
satisfies that $z_{\frac{1}{2},x,j}\ge z_{-\frac{1}{2},x,i}+\frac{1}{2}$
for all $i,j\in[m],x\in\mathcal{X}$ such that $z_{-\frac{1}{2},x,i}>\frac{-1}{2\rho},z_{\frac{1}{2},x,j}<\frac{1}{2\rho}$.
\\
 4. For even $m\ge8$, there is an optimal solution $P_{SXZ}$ to
the maximization problem in \eqref{eq:generalboundsimple-2} such
that $P_{SXZ}(-s,-x,-z)=P_{SXZ}(s,x,z)$ for all $(s,x,z)$, which
either satisfies the LICQ or can be expressed as 
\begin{equation}
P_{SXZ}(s,x,z)=\begin{cases}
\frac{1+2\beta}{4} & (s,x,z)=(-\frac{1}{2},-1,z_{1})\\
\frac{1-2\beta}{4} & (s,x,z)=(\frac{1}{2},-1,z_{2})\\
\frac{1-2\beta}{4} & (s,x,z)=(-\frac{1}{2},1,-z_{2})\\
\frac{1+2\beta}{4} & (s,x,z)=(\frac{1}{2},1,-z_{1})
\end{cases}\label{eq:-50}
\end{equation}
for some $z_{1},z_{2}$ such that $-\frac{1}{2\rho}\leq z_{1}\leq z_{2}\leq\frac{1}{2\rho}$. 
\end{prop}

Define 
\begin{equation}
\overline{\Upsilon}_{\rho}:=\max_{\beta\in[0,\frac{1}{2}]}\max_{(z_{1},z_{2})\in\check{\mathcal{Z}}}h(\beta,z_{1},z_{2}),\label{eq:-49}
\end{equation}
where 
\[
\check{\mathcal{Z}}:=\{(z_{1},z_{2}):-\frac{1}{2\rho}\leq z_{1}\leq z_{2}\leq\frac{1}{2\rho},0\leq p\leq\frac{1}{4}+\frac{\beta}{2},0\leq q\leq\frac{1}{4}-\frac{\beta}{2}\},
\]
and 
\[
h(\beta,z_{1},z_{2}):=(1-2p-2q)\Phi(0)+2p\Phi(\frac{1}{2}+\rho z_{1})+2q\Phi(\frac{1}{2}+\rho z_{2}),
\]
with 
\begin{align}
 & p=\frac{(1-\rho)(1+\rho-4\rho^{2}\omega(\beta)+2\beta(1+2\rho z_{2}-\rho^{2}))}{4(1+2\rho z_{1})(1+\rho z_{2}-\rho z_{1}-\rho^{2})},\label{eq:p-2}\\
 & q=\frac{(1-\rho)(1+\rho-4\rho^{2}\omega(\beta)-2\beta(1-2\rho z_{1}-\rho^{2}))}{4(1-2\rho z_{2})(1+\rho z_{2}-\rho z_{1}-\rho^{2})}.\label{eq:q-2}
\end{align}

\begin{thm}
\label{thm:symmetric-1} If $\Phi$ is symmetric and continuous on
$[0,1]$ and differentiable on $(0,1)$ whose derivative $\Phi'$
is increasing and continuous $(0,1)$, and strictly concave on $(0,\frac{1}{2}]$,
then $\mathbf{MaxStab}_{\Phi}(\frac{1}{2})\leq\overline{\Upsilon}_{\rho}.$ 
\end{thm}

The proof of Theorem \ref{thm:symmetric-1} is provided in Section
\ref{sec:Proof-of-Theorem-symmetric-1} which is similar to that of
Theorem \ref{thm:symmetric}.

We now focus on the case $\Phi=\Phi_{1}^{\mathrm{sym}}$. Numerical
results show that for the case $\Phi=\Phi_{1}^{\mathrm{sym}}$, $\overline{\Upsilon}_{\rho}$
in Theorem \ref{thm:symmetric-1} (with $\omega(\beta)$ taken as
the minimum of the upper bounds in \eqref{eq:-71} and \eqref{eq:-21})
is (almost) attained by dictator functions for $0\le\rho\le0.83$,
which means that the threshold $\rho^{*}$ in Corollary \ref{cor:CK_Conjecture}
can be further improved to a value around $0.83$. It remains to find
a proof for this observation in the future. We are also interested
in introducing new techniques to attack the case with $\rho>0.83$,
since the optimality of dictator functions for this case seemingly
cannot be proven by our present method.

In following sections, we provide proofs for the results stated above.

\section{\label{sec:Proof-of-Proposition}Proof of Proposition \ref{prop:property} }

\subsection{Statement 2}

We use a perturbation method. Suppose $P_{Z|S}$ is an optimal solutions
to \eqref{eq:generalbound} such that $\mathbb{E}[Z^{2}]<\mathbb{E}[SZ]$.
We first assume that $a+\rho z=0$ or $1$ for all $z$ such that
$P_{Z}(z)>0$, which is equivalent to $Z\in\{-a/\rho,(1-a)/\rho\}$
(with probability one). By \eqref{eq:con2}, we know that $P_{Z}(z)=\begin{cases}
1-a & z=-a/\rho\\
a & z=(1-a)/\rho
\end{cases}$. From this, we have $\mathbb{E}[SZ]\le\sqrt{\mathbb{E}[Z^{2}]\mathbb{E}[S^{2}]}=\rho\mathbb{E}[Z^{2}]$
which contradicts with \eqref{eq:con3}. Hence, our assumption is
false, or equivalently, $0<a+\rho z^{*}<1$ for some $z^{*}$ such
that $P_{Z}(z^{*})>0$.

Let $s^{*}$ be such that $P_{SZ}(s^{*},z^{*})>0$. We next construct
a new conditional distribution $Q_{Z|S}$ by setting $Q_{Z|S}(z|s)=P_{Z|S}(z|s)$
for all $(s,z)\neq(s^{*},z^{*})$, and 
\[
Q_{Z|S}(z^{*}-\delta|s^{*})=Q_{Z|S}(z^{*}+\delta|s^{*})=\frac{1}{2}P_{Z|S}(z^{*}|s^{*}),
\]
where $\delta>0$ is small enough such that $0<a+\rho(z^{*}\pm\delta)<1$.
By the choice of $\delta$, \eqref{eq:con1} holds. It is easy to
see that \eqref{eq:con2} still holds. Furthermore, observe that the
RHS of \eqref{eq:con3} remains unchanged, and the LHS of \eqref{eq:con3}
is continuous in $\delta$. Hence for sufficiently small but positive
$\delta$, \eqref{eq:con3} still holds. Since $\Phi$ is strictly
convex, $\mathbb{E}[\Phi(a+\rho Z)]$ increases after replacing $P_{Z|S}$
with $Q_{Z|S}$. This contradicts with the optimality of $P_{Z|S}$.
This completes the proof of Statement 2.

\subsection{Statement 3}

We continue to use a perturbation method to prove Statement 3. Let
$(z_{s,i},p_{s,i})_{(s,i)\in\mathcal{S}\times[m]}$ be an optimal
solution to the maximization in \eqref{eq:generalboundsimple}. Let
$P_{SZ}$ be the joint distribution induced by $(z_{s,i},p_{s,i})_{(s,i)\in\mathcal{S}\times[m]}$.
By Statement 2, $\mathbb{E}_{P}[Z^{2}]=\mathbb{E}_{P}[SZ]$ under
$P_{SZ}$.

Suppose that $z_{-a,i}>z_{1-a,j}$ for some $i,j\in[m]$. Then, for
a sufficiently small $\epsilon>0$, we define a new distribution $Q_{SZ}$
by replacing $(z_{-a,i},p_{-a,i})\leftarrow(z_{-a,i},p_{-a,i}-\epsilon)$,
$(z_{1-a,j},p_{1-a,j})\leftarrow(z_{1-a,j},p_{1-a,j}-\epsilon)$ and
introducing new points $(z_{-a,m+1},p_{-a,m+1})\leftarrow(z_{1-a,j},\epsilon)$,
$(z_{1-a,m+1},p_{1-a,m+1})\leftarrow(z_{-a,i},\epsilon)$. We do not
change other parameters. For this new distribution, it is easily seen
that the marginal distributions of $Z$ and $S$ are unchanged. So,
constraints in \eqref{eq:con1-1}-\eqref{eq:con3-1} are still satisfied
by $Q_{SZ}$, and the value of the objective function induced by $Q_{SZ}$
remains unchanged. As for constraint in \eqref{eq:con4-1}, $\mathbb{E}[Z^{2}]$
remains the same, and $\mathbb{E}_{Q}[SZ]-\mathbb{E}_{P}[SZ]=\epsilon(z_{-a,i}-z_{1-a,j})>0.$
So, $\mathbb{E}_{Q}[SZ]>\mathbb{E}_{P}[SZ]=\mathbb{E}_{P}[Z^{2}]=\mathbb{E}_{Q}[Z^{2}]$.
Hence, \eqref{eq:con4-1} is satisfied by $Q_{SZ}$, which further
implies that $Q_{SZ}$ is an optimal solution to \eqref{eq:generalboundsimple}.
However, by Statement 2, $\mathbb{E}_{Q}[SZ]=\mathbb{E}_{Q}[Z^{2}]$
should hold, which leads to a contradiction. Therefore, $z_{-a,i}\le z_{1-a,j}$
for all $i,j\in[m]$.

Suppose that $z_{-a,i}\le z_{1-a,j}<z_{-a,i}+\frac{1}{2}$ for some
$i,j\in[m]$ such that $z_{-a,i}>\frac{-a}{\rho},\,z_{1-a,j}<\frac{1-a}{\rho}$.
Then, we define a new distribution $Q_{SZ}$ by replacing $z_{-a,i}\leftarrow z_{-a,i}-\epsilon,\;z_{1-a,j}\leftarrow z_{1-a,j}+\epsilon'$
for some $\epsilon,\epsilon'>0$ such that $p_{-a,i}\epsilon=p_{1-a,j}\epsilon'$.
By definition, $Q_{SZ}$ does not change $\mathbb{E}[Z]$, but it
enlarges $\mathbb{E}[SZ]-\mathbb{E}[Z^{2}]$ for small enough $\epsilon,\epsilon'>0$
since 
\begin{align*}
 & (\mathbb{E}_{Q}[SZ]-\mathbb{E}_{Q}[Z^{2}])-(\mathbb{E}_{P}[SZ]-\mathbb{E}_{P}[Z^{2}])\\
 & =(p_{-a,i}\epsilon a+p_{1-a,j}\epsilon'\bar{a})-(-p_{-a,i}\epsilon(2z_{-a,i}-\epsilon)+p_{1-a,j}\epsilon'(2z_{1-a,j}+\epsilon'))\\
 & =p_{-a,i}\epsilon(1-2(z_{1-a,j}-z_{-a,i})-(\epsilon+\epsilon')),
\end{align*}
which is positive for small enough $\epsilon,\epsilon'>0$. Hence,
$Q_{SZ}$ is feasible. On the other hand, $Q_{SZ}$ also enlarges
the objective function. This is because, by denoting $\hat{\Phi}(z)=\Phi(a+\rho z)$,
\begin{align*}
 & \mathbb{E}_{Q}[\hat{\Phi}(Z)]-\mathbb{E}_{P}[\hat{\Phi}(Z)]\\
 & =p_{-a,i}\hat{\Phi}(z_{-a,i}-\epsilon)+p_{1-a,j}\hat{\Phi}(z_{1-a,j}+\epsilon')-p_{-a,i}\hat{\Phi}(z_{-a,i})-p_{1-a,j}\hat{\Phi}(z_{1-a,j})\\
 & =(p_{1-a,j}\epsilon')\frac{\hat{\Phi}(z_{1-a,j}+\epsilon')-\hat{\Phi}(z_{1-a,j})}{\epsilon'}-(p_{-a,i}\epsilon)\frac{\hat{\Phi}(z_{-a,i})-\hat{\Phi}(z_{-a,i}-\epsilon)}{\epsilon}
\end{align*}
which is positive since $p_{1-a,j}\epsilon'=p_{-a,i}\epsilon$, and
by the convexity of $\Phi$, $\frac{\hat{\Phi}(z_{1-a,j}+\epsilon')-\hat{\Phi}(z_{1-a,j})}{\epsilon'}>\frac{\hat{\Phi}(z_{-a,i})-\hat{\Phi}(z_{-a,i}-\epsilon)}{\epsilon}$
(note that $z_{1-a,j}\ge z_{-a,i}$). Hence, $Q_{SZ}$ induces a larger
value of the objective function than $P_{SZ}$, which contradicts
with the optimality of $P_{SZ}$. Hence, $z_{1-a,j}\ge z_{-a,i}+\frac{1}{2}$
for all $i,j\in[m]$ such that $z_{-a,i}>\frac{-a}{\rho},\,z_{1-a,j}<\frac{1-a}{\rho}$.

If $\{z_{1-a,j}\}_{j\in[m]}=\{\frac{1-a}{\rho}\}$, then by \eqref{eq:con3-3},
$\{z_{-a,i}\}_{i\in[m]}=\{\frac{-a}{\rho}\}$. Hence, we have $Z=\frac{S}{\rho}$,
and hence, $\mathbb{E}[Z^{2}]=\frac{1}{\rho}\mathbb{E}[SZ]$, which
contradicts with $\mathbb{E}[Z^{2}]\le\mathbb{E}[SZ]$. Therefore,
$z_{1-a,j}<\frac{1-a}{\rho}$ for some $j\in[m]$. Similarly, one
can show that $z_{-a,i}>\frac{-a}{\rho}$ for some $i\in[m]$.

\subsection{Statement 4 }

Suppose that LICQ is not satisfied. Denote $(z_{s,i},p_{s,i})_{(s,i)\in\mathcal{S}\times[m]}$
as an optimal solution to the maximization in \eqref{eq:generalboundsimple}.
Here and subsequently, we assume the first $m$ components are indexed
by $\{-a\}\times[m]$ and the last $m$ indexed by $\{1-a\}\times[m]$.
Denote $\mathcal{I}:=\{(s,i):a+\rho z_{s,i}=0\}$ and $\mathcal{J}:=\{(s,i):a+\rho z_{s,i}=1\}$.
By Statement 3, $\mathcal{I}\subseteq\{-a\}\times[m]$ and $\mathcal{J}\subseteq\{1-a\}\times[m]$.

The gradients of the active inequality constraints (including \eqref{eq:con3}
by Statement 2) and the gradients of the equality constraints constitute
the $(|\mathcal{I}|+|\mathcal{J}|+4)\times4m$ matrix 
\begin{equation}
G:=\begin{bmatrix}\rho\mathbf{I}_{\mathcal{I}} & \mathbf{0}_{1\times2m}\\
-\rho\mathbf{I}_{\mathcal{J}} & \mathbf{0}_{1\times2m}\\
\mathbf{0}_{1\times2m} & (\mathbf{I}_{1\times m},\mathbf{0}_{1\times m})\\
\mathbf{0}_{1\times2m} & (\mathbf{0}_{1\times m},\mathbf{I}_{1\times m})\\
(p_{s,i})_{s,i} & (z_{s,i})_{s,i}\\
(p_{s,i}(2z_{s,i}-s))_{s,i} & (z_{s,i}^{2}-sz_{s,i})_{s,i}
\end{bmatrix}.\label{eq:G-3}
\end{equation}
Here, $\mathbf{I}_{\mathcal{I}}$ denotes the $\{0,1\}$-matrix of
size $|\mathcal{I}|\times2m$ with each row containing exactly one
``$1$'', and the ``$1$'' at the $r$-th row is located at the
column indexed by the $r$-th element in $\mathcal{I}$. Obviously,
the last $m$ columns of $\mathbf{I}_{\mathcal{I}}$ consist of zeros.
The matrix $\mathbf{I}_{\mathcal{J}}$ is defined similarly, and hence,
the first $m$ columns of $\mathbf{I}_{\mathcal{J}}$ consist of zeros.
The assumption that LICQ is not satisfied implies that 
\[
\hat{G}:=\begin{bmatrix}(p_{s,i})_{(s,i)\in(\mathcal{I}\cup\mathcal{J})^{c}} & (z_{s,i}-z_{s,1})_{(s,i)\in\mathcal{S}\times[2:m]}\\
(p_{s,i}(2z_{s,i}-s))_{(s,i)\in(\mathcal{I}\cup\mathcal{J})^{c}} & ((z_{s,i}-z_{s,1})(z_{s,i}+z_{s,1}-s))_{(s,i)\in\mathcal{S}\times[2:m]}
\end{bmatrix}
\]
is not of full rank. We next prove this is impossible.

Since the submatrix $\begin{bmatrix}(p_{s,i})_{(s,i)\in(\mathcal{I}\cup\mathcal{J})^{c}}\\
(p_{s,i}(2z_{s,i}-s))_{(s,i)\in(\mathcal{I}\cup\mathcal{J})^{c}}
\end{bmatrix}$ of $\hat{G}$ is not of full rank, we know that $z_{-a,i}=z_{1}$
or $-a/\rho$ and $z_{1-a,i}=z_{2}$ or $(1-a)/\rho$ for all $i\in[m]$
and some $z_{1},z_{2}$ such that $-a-2z_{1}=1-a-2z_{2}$. Furthermore,
since the submatrix $\begin{bmatrix}(p_{-a,i})_{(-a,i)\in(\mathcal{I}\cup\mathcal{J})^{c}} & (z_{-a,i}-z_{-a,1})_{i\in[2:m]}\\
(p_{-a,i}(2z_{1}+a))_{(-a,i)\in(\mathcal{I}\cup\mathcal{J})^{c}} & ((z_{-a,i}-z_{-a,1})(z_{-a,i}+z_{-a,1}+a))_{i\in[2:m]}
\end{bmatrix}$ of $\hat{G}$ is also not of full rank, we know that $z_{-a,i}$'s
are identical for all $i$, and combining this with Statement 3 implies
that $z_{-a,i}=z_{1}$ for all $i\in[m]$. Similarly, $z_{1-a,i}=z_{2}$
for all $i\in[m]$. Additionally, $z_{1},z_{2}$ should satisfy $\mathbb{E}[Z]=0,\mathbb{E}[Z^{2}]=\mathbb{E}[SZ]$,
and hence, we have $Z=0$ or $Z=S$. This contradicts with $-a-2z_{1}=1-a-2z_{2}$,
which implies that LICQ is satisfied.

\section{\label{sec:Proof-of-Theorem-symmetric}Proof of Theorem \ref{thm:symmetric}}

The Lagrangian of the optimization problem in \eqref{eq:generalboundsimple}
is 
\begin{align*}
\mathcal{L}:= & \sum_{s,i}p_{s,i}\Big\{\Phi(a+\rho z_{s,i})+\theta_{s,i}(a+\rho z_{s,i})+\theta_{s,i}'(1-(a+\rho z_{s,i}))+\lambda(sz_{s,i}-z_{s,i}^{2})+\eta z_{s,i}\Big\}\\
 & \qquad+\sum_{s}\mu_{s}(\sum_{i}p_{s,i}-P_{S}(s)).
\end{align*}
Since LICQ is satisfied, the KKT theorem, e.g., \cite[Theorem 5.3.1 ]{bazaraa2013nonlinear},
is valid, and then we obtain the following first-order necessary conditions
for (local) optimal solutions: 
\begin{align}
 & \frac{\partial\mathcal{L}}{\partial z_{s,i}}=p_{s,i}[\rho(\Phi'(a+\rho z_{s,i})+\theta_{s,i}-\theta_{s,i}')+\lambda(s-2z_{s,i})+\eta]=0\label{eq:dLagrangian}\\
 & \frac{\partial\mathcal{L}}{\partial p_{s,i}}=\Phi(a+\rho z_{s,i})+\theta_{s,i}(a+\rho z_{s,i})+\theta_{s,i}'(1-(a+\rho z_{s,i}))+\lambda(sz_{s,i}-z_{s,i}^{2})+\eta z_{s,i}+\mu_{s}\\
 & \qquad=\Phi(a+\rho z_{s,i})+\lambda(sz_{s,i}-z_{s,i}^{2})+\eta z_{s,i}+\mu_{s}=0\\
 & 0\leq a+\rho z_{s,i}\leq1,\forall(s,i)\in\mathcal{S}\times[m]\label{eq:-78-2-2}\\
 & \sum_{s,i}p_{s,i}z_{s,i}=0\\
 & \sum_{s,i}p_{s,i}(sz_{s,i}-z_{s,i}^{2})=0\label{eq:-70-2-2}\\
 & \sum_{i}p_{s,i}=P_{S}(s),\forall s\in\mathcal{S}\\
 & p_{s,i}>0,\forall(s,i)\in\mathcal{S}\times[m]\\
 & \theta_{s,i}(a+\rho z_{s,i})=0\\
 & \theta_{s,i}'(1-a-\rho z_{s,i})=0\\
 & \lambda\geq0,\overrightarrow{\theta},\overrightarrow{\theta'}\geq\overrightarrow{0},\label{eq:-17}
\end{align}
where the equality in \eqref{eq:-70-2-2} follows from Statement 2
of Proposition \ref{prop:property}. Here, $\overrightarrow{\theta},\overrightarrow{\theta'}\geq\overrightarrow{0}$
denotes all the components of $\overrightarrow{\theta},\overrightarrow{\theta'}$
are nonnegative.

For $(s,i)$ such that $-a/\rho<z_{s,i}<(1-a)/\rho$, we have $\theta_{s,i}=\theta_{s,i}'=0$.
Denote 
\begin{align}
g(s,z) & :=\rho\Phi'(a+\rho z)+\lambda(s-2z)+\eta,\label{eq:g}\\
G(s,z) & :=\int_{0}^{z}g(s,t)\mathrm{d}t.\nonumber 
\end{align}
Hence, for $(s,i)$ such that $-a/\rho<z_{s,i}<(1-a)/\rho$, it holds
that 
\begin{align}
\frac{1}{p_{s,i}}\frac{\partial\mathcal{L}}{\partial z_{s,i}} & =g(s,z_{s,i})=0,\label{eq:geq}
\end{align}
and moreover, for all $(s,i)$, it holds that 
\begin{align}
\frac{\partial\mathcal{L}}{\partial p_{s,i}} & =G(s,z_{s,i})+\Phi(a)+\mu_{s}=0.\label{eq:Geq}
\end{align}
Equations \eqref{eq:geq} and \eqref{eq:Geq} imply that given $s$,
for all $z\in\{z_{s,i}\}_{i\in[m]}$, $g(s,z)$ is always equal to
$0$ and $G(s,z)$ remains the same.

\begin{figure}
\centering \subfloat[{{\label{fig:plots-1} $\Phi'$ is concave on $(0,\frac{1}{2}]$.}}]{ \tikzset{every picture/.style={line width=0.75pt}} 
\begin{tikzpicture}[x=0.75pt,y=0.75pt,yscale=-1,xscale=1,scale=.7] 
\draw    (174.49,488.96) .. controls (174.24,357.26) and (296.49,448.96) .. (296.53,314.32) ; 
\draw    (151,403.29) -- (349.27,403.29) ;
\draw [shift={(351.27,403.29)}, rotate = 180] [color={rgb, 255:red, 0; green, 0; blue, 0 }  ][line width=0.75]    (10.93,-3.29) .. controls (6.95,-1.4) and (3.31,-0.3) .. (0,0) .. controls (3.31,0.3) and (6.95,1.4) .. (10.93,3.29)   ; 
\draw    (313.27,313.62) -- (158.93,467.96) ; 
\draw  [dash pattern={on 0.84pt off 2.51pt}]  (296.53,314.32) -- (296.53,403.62) ; 
\draw    (314.5,334.39) -- (162.93,485.96) ; 
\draw    (195.27,394.12) .. controls (206.67,400.77) and (207.24,400.2) .. (215.61,408.25) ;
\draw [shift={(217,409.59)}, rotate = 224.22] [color={rgb, 255:red, 0; green, 0; blue, 0 }  ][line width=0.75]    (8.74,-2.63) .. controls (5.56,-1.12) and (2.65,-0.24) .. (0,0) .. controls (2.65,0.24) and (5.56,1.12) .. (8.74,2.63)   ; 
\draw  [dash pattern={on 0.84pt off 2.51pt}]  (174.49,404.12) -- (174.49,488.96) ; 
\draw    (155.27,438.12) .. controls (166.31,438.12) and (171.42,440.66) .. (178.19,443.77) ;
\draw [shift={(180,444.59)}, rotate = 204.17] [color={rgb, 255:red, 0; green, 0; blue, 0 }  ][line width=0.75]    (8.74,-2.63) .. controls (5.56,-1.12) and (2.65,-0.24) .. (0,0) .. controls (2.65,0.24) and (5.56,1.12) .. (8.74,2.63)   ; 
\draw    (271.27,322.57) .. controls (282.85,324.5) and (278.6,322.7) .. (292.42,329.78) ;
\draw [shift={(294,330.59)}, rotate = 207.03] [color={rgb, 255:red, 0; green, 0; blue, 0 }  ][line width=0.75]    (8.74,-2.63) .. controls (5.56,-1.12) and (2.65,-0.24) .. (0,0) .. controls (2.65,0.24) and (5.56,1.12) .. (8.74,2.63)   ;
\draw (341,406.29) node [anchor=north west][inner sep=0.75pt]    {$z$}; 
\draw (177,375.96) node [anchor=north west][inner sep=0.75pt]    {$z_{2}$}; 
\draw (251.13,403.29) node [anchor=north west][inner sep=0.75pt]    {$( 1-a) /\rho $}; 
\draw (129.13,400.29) node [anchor=north west][inner sep=0.75pt]    {$-a/\rho $}; 
\draw (137,426.96) node [anchor=north west][inner sep=0.75pt]    {$z_{1}$}; 
\draw (253,308.96) node [anchor=north west][inner sep=0.75pt]    {$z_{3}$}; 
\draw (225,288) node [anchor=north west][inner sep=0.75pt]    {$\rho \Phi '( a+\rho z)$}; 
\draw (313,300.23) node [anchor=north west][inner sep=0.75pt]    {$\lambda ( 2z+a) -\eta $}; 
\draw (314,325.23) node [anchor=north west][inner sep=0.75pt]    {$\lambda ( 2z+a-1) -\eta $};
\end{tikzpicture}

}\subfloat[{{\label{fig:plots-2} $\Phi'$ is convex on $(0,\frac{1}{2}]$.}}]{ \tikzset{every picture/.style={line width=0.75pt}} 
\begin{tikzpicture}[x=0.75pt,y=0.75pt,yscale=-1,xscale=1,scale=.7] 
\draw    (208.49,708.47) .. controls (320.49,708.47) and (214.75,620.5) .. (330.53,620.9) ; 
\draw    (185,665.22) -- (383.27,665.22) ;
\draw [shift={(385.27,665.22)}, rotate = 180] [color={rgb, 255:red, 0; green, 0; blue, 0 }  ][line width=0.75]    (10.93,-3.29) .. controls (6.95,-1.4) and (3.31,-0.3) .. (0,0) .. controls (3.31,0.3) and (6.95,1.4) .. (10.93,3.29)   ; 
\draw    (339.27,585.55) -- (191.93,732.89) ; 
\draw  [dash pattern={on 0.84pt off 2.51pt}]  (330.53,576.25) -- (330.53,665.55) ; 
\draw    (340.27,602.39) -- (194.93,747.89) ; 
\draw  [dash pattern={on 0.84pt off 2.51pt}]  (208.49,666.05) -- (208.49,750.89) ; 
\draw    (281.27,608.05) .. controls (292.73,615.69) and (286.86,611.47) .. (296.53,620.2) ;
\draw [shift={(298,621.52)}, rotate = 221.75] [color={rgb, 255:red, 0; green, 0; blue, 0 }  ][line width=0.75]    (8.74,-2.63) .. controls (5.56,-1.12) and (2.65,-0.24) .. (0,0) .. controls (2.65,0.24) and (5.56,1.12) .. (8.74,2.63)   ; 
\draw    (201.27,694.05) .. controls (212.73,701.69) and (206.86,697.47) .. (216.53,706.2) ;
\draw [shift={(218,707.52)}, rotate = 221.75] [color={rgb, 255:red, 0; green, 0; blue, 0 }  ][line width=0.75]    (8.74,-2.63) .. controls (5.56,-1.12) and (2.65,-0.24) .. (0,0) .. controls (2.65,0.24) and (5.56,1.12) .. (8.74,2.63)   ; 
\draw    (246.27,645.57) .. controls (257.85,647.5) and (253.6,645.7) .. (267.42,652.78) ;
\draw [shift={(269,653.59)}, rotate = 207.03] [color={rgb, 255:red, 0; green, 0; blue, 0 }  ][line width=0.75]    (8.74,-2.63) .. controls (5.56,-1.12) and (2.65,-0.24) .. (0,0) .. controls (2.65,0.24) and (5.56,1.12) .. (8.74,2.63)   ;
\draw (375,668.22) node [anchor=north west][inner sep=0.75pt]    {$z$}; 
\draw (285.13,665.22) node [anchor=north west][inner sep=0.75pt]    {$( 1-a) /\rho $}; 
\draw (163.13,662.22) node [anchor=north west][inner sep=0.75pt]    {$-a/\rho $}; 
\draw (184,682.89) node [anchor=north west][inner sep=0.75pt]    {$z_{1}$}; 
\draw (261,593.89) node [anchor=north west][inner sep=0.75pt]    {$z_{3}$}; 
\draw (339,593.23) node [anchor=north west][inner sep=0.75pt]    {$\lambda ( 2z+a-1) -\eta $}; 
\draw (338,571.23) node [anchor=north west][inner sep=0.75pt]    {$\lambda ( 2z+a) -\eta $}; 
\draw (192.53,613.9) node [anchor=north west][inner sep=0.75pt]    {$\rho \Phi '( a+\rho z)$}; 
\draw (228,631.96) node [anchor=north west][inner sep=0.75pt]    {$z_{2}$};
\end{tikzpicture}

}\caption{\label{fig:plots} Plots of $z\protect\mapsto\rho\Phi'(a+\rho z)$,
$z\protect\mapsto\lambda(2z+a)-\eta$, and $z\protect\mapsto\lambda(2z+a-1)-\eta$.}
\end{figure}

\subsection{\label{subsec:Statement-1}Statement 1}

By assumption, $\Phi$ is symmetric w.r.t. $\frac{1}{2}$ and $\Phi'$
is increasing, continuous on $(0,1)$, and strictly concave on $(0,\frac{1}{2}]$.
Hence, for each $s$, $g(s,z)=0$ with $z$ unknown has at most three
distinct solutions. For the case of three distinct solutions, denote
the solutions as $z_{1},z_{2},z_{3}$ such that $-a/\rho=:z_{0}<z_{1}<z_{2}<z_{3}<z_{4}:=(1-a)/\rho$.
Obviously, for these solutions, the values of $G(s,z_{i})$ and $G(s,z_{i+1})$
for $i\in\{1,2\}$ are different (note that $G$ is an antiderivative
of $g$). Hence, given $s$, $z_{i}$ and $z_{i+1}$ cannot be solutions
of \eqref{eq:Geq} at the same time. This is also true when $g(s,z)=0$
(with $z$ unknown) has two distinct solutions.

On the other hand, by Statement 3 of Proposition \ref{prop:property},
for each $s$, $\{z_{s,i}\}_{i\in[m]}$ contains at least one solution
of \eqref{eq:geq} and \eqref{eq:Geq}. Moreover, the solutions to
$g(-a,z)=0$ are the intersection of the curve $z\mapsto\rho\Phi'(a+\rho z)$
and the line $z\mapsto\lambda(2z+a)-\eta$ (with a positive slope),
and the solutions to $g(1-a,z)=0$ are the intersection points of
the same curve $z\mapsto\rho\Phi'(a+\rho z)$ and another line $z\mapsto\lambda(2z+a-1)-\eta$
which is parallel to the previous one. See Fig. \ref{fig:plots-1}.

From the points above, we claim that for each $s$, $\{z_{s,i}\}_{i\in[m]}$
cannot contain two distinct solutions of \eqref{eq:geq} and \eqref{eq:Geq}.
This is because, for example, for $s=-a$, if $\{z_{s,i}\}_{i\in[m]}$
contains two distinct solutions of \eqref{eq:geq} and \eqref{eq:Geq},
then \eqref{eq:geq} and \eqref{eq:Geq} must have three distinct
solutions, and the smallest and the largest among them, denoted by
$z_{1},z_{3}$ respectively, must be contained in $\{z_{s,i}\}_{i\in[m]}$.
In this case, all the solutions of \eqref{eq:geq} and \eqref{eq:Geq}
with $s=1-a$ are smaller than $z_{3}$ (by the observation in the
paragraph above), and hence cannot be contained in $\{z_{1-a,i}\}_{i\in[m]}$
(by the first part of Statement 3 of Proposition \ref{prop:property}).
This contradicts with (the second part of) Statement 3 of Proposition
\ref{prop:property}. Hence, the claim holds for $s=-a$. For $s=1-a$,
the claim can be proven similarly.

By the claim above, 
\begin{equation}
\{z_{-a,i}\}_{i\in[m]}\subseteq\{-a/\rho,\hat{z}_{1}\}\qquad\textrm{ and }\qquad\{z_{1-a,i}\}_{i\in[m]}\subseteq\{(1-a)/\rho,\hat{z}_{2}\},\label{eq:-36}
\end{equation}
where $\hat{z}_{1}$ is a solution of \eqref{eq:geq} and \eqref{eq:Geq}
with $s=-a$, and $\hat{z}_{2}$ is a solution of \eqref{eq:geq}
and \eqref{eq:Geq} with $s=1-a$. By Statement 3 of Proposition \ref{prop:property},
$\hat{z}_{2}-\hat{z}_{1}\geq\frac{1}{2}$. This means, optimal solutions
$P_{SZ}$ to the optimization problem in \eqref{eq:generalboundsimple}
satisfy 
\[
P_{SZ}(s,z)=\begin{cases}
1-a-p & s=-a,z=\frac{-a}{\rho}\\
p & s=-a,z=\hat{z}_{1}\\
q & s=1-a,z=\hat{z}_{2}\\
a-q & s=1-a,z=\frac{1-a}{\rho}
\end{cases}
\]
for some $(p,q,\hat{z}_{1},\hat{z}_{2})$ such that $0\le p\leq1-a,0\le q\leq a,-\frac{a}{\rho}<\hat{z}_{1}\le\hat{z}_{2}<\frac{1-a}{\rho}$,
$\hat{z}_{2}-\hat{z}_{1}\geq\frac{1}{2}$, and 
\begin{align}
 & \mathbb{E}Z=0,\qquad\mathbb{E}[SZ]=\mathbb{E}[Z^{2}].\label{eq:-16-2}
\end{align}
Solving equations in \eqref{eq:-16-2} with respect to unknowns $(z_{1},z_{2})$,
we obtain $p,q$ given in \eqref{eq:-7} and \eqref{eq:-8}. Since
$\hat{z}_{2}\geq\hat{z}_{1}$, $0\le p\le1-a$, and $0\le q\le a$,
we know that 
\begin{align*}
 & \hat{z}_{2}-\hat{z}_{1}\geq\frac{(1-\rho)\hat{z}_{2}}{1-a-\rho\hat{z}_{2}}\\
 & \hat{z}_{2}-\hat{z}_{1}\geq-\frac{(1-\rho)\hat{z}_{1}}{a+\rho\hat{z}_{1}}.
\end{align*}
Hence, the first part of Statement 1 holds.

We next prove the second part of Statement 1, in which $a=1/2$. Denote
$\delta_{x}$ as the Dirac measure at $x$. For an optimal solution
$z_{s,i},p_{s,i},(s,i)\in\mathcal{S}\times[m]$ to the optimization
problem in \eqref{eq:generalboundsimple} which must satisfy \eqref{eq:-36},
we now construct a new distribution $Q_{SZ}=\sum_{s,i}\frac{1}{2}p_{s,i}\cdot(\delta_{s,z_{s,i}}+\delta_{-s,-z_{s,i}})$
which is still optimal for the optimization in \eqref{eq:generalboundsimple}
with $m$ replaced by $2m$. Obviously, $Q_{SZ}$ should also satisfy
\eqref{eq:-36} since the above arguments still work when $m$ is
replaced by $2m$. Hence, $\hat{z}_{1}+\hat{z}_{2}=0$ holds.

\subsection{Statement 2}

We next consider the case in which $\Phi'$ is increasing, continuous
on $(0,1)$, and strictly convex on $(0,\frac{1}{2}]$. Let $z_{s,i},p_{s,i},(s,i)\in\mathcal{S}\times[m]$
be an optimal solution to the maximization problem in \eqref{eq:generalboundsimple}.
We next derive necessary conditions for the optimality of this solution.
For this case, the equation \eqref{eq:geq} still has at most three
distinct solutions.

We first consider the case $g(-a,-a/\rho)>0$ with $g$ defined in
\eqref{eq:g}. For this case, at $(s,z)=(-a,-a/\rho)$, $\left.\frac{1}{p_{s,i}}\frac{\partial\mathcal{L}}{\partial z_{s,i}}\right|_{(-a,-a/\rho)}=g(-a,-a/\rho)+\theta>0$.
Hence, $\{z_{-a,i}\}_{i\in[m]}$ does not contain $-a/\rho$, and
does contain at most two distinct solutions of \eqref{eq:geq} and
\eqref{eq:Geq} with $s=-a$ (by arguments similar to those in the
previous subsection).

We now claim that if $g(-a,-a/\rho)>0$, and meanwhile, $\{z_{-a,i}\}_{i\in[m]}$
contains exactly two distinct solutions of \eqref{eq:geq} and \eqref{eq:Geq}
(with $s=-a$), then $\{z_{-a,i}\}_{i\in[m]}=\{\tilde{z},b-\tilde{z}\}$
with $b=\frac{1-2a}{\rho}$ for some $\tilde{z}>-a/\rho$, and moreover,
$\{z_{1-a,i}\}_{i\in[m]}$ is a singleton. We next prove this claim.
If $\{z_{-a,i}\}_{i\in[m]}$ contains exactly two distinct solutions
of \eqref{eq:geq} and \eqref{eq:Geq} (with $s=-a$), then $g(-a,z)=0$
has three distinct solutions, denoted by $-a/\rho<z_{1}<z_{2}<z_{3}\le(1-a)/\rho$,
and the solutions contained in $\{z_{-a,i}\}_{i\in[m]}$ are $z_{1},z_{3}$.
See Fig. \ref{fig:plots-2}. Moreover, by the facts that $G$ is an
antiderivative of $g$ and $\Phi'$ is symmetric w.r.t. $1/2$, it
holds that $G(-a,z_{1})=G(-a,z_{3})$ only if $a+\rho z_{1}+a+\rho z_{3}=1$
(since only in this case, the two areas enclosed by the curve $z\mapsto\rho\Phi'(a+\rho z)$
and the line $z\mapsto\lambda(2z+a)-\eta$ have the same size). In
this case, by Statement 3 of Proposition \ref{prop:property}, for
$s=1-a$, $\{z_{1-a,i}\}_{i\in[m]}$ only contain $(1-a)/\rho$ and/or
the largest solution $\hat{z}_{3}$ to \eqref{eq:geq} and \eqref{eq:Geq}
with $s=1-a$. If $\hat{z}_{3}<(1-a)/\rho$, obviously, $G(-a,\hat{z}_{3})>G(1-a,(1-a)/\rho)$,
and hence, $\hat{z}_{3}$ and $(1-a)/\rho$ cannot be solutions to
\eqref{eq:Geq} at the same time. Hence, $\{z_{1-a,i}\}_{i\in[m]}$
is a singleton and $\{z_{-a,i}\}_{i\in[m]}=\{\tilde{z},b-\tilde{z}\}$
for some $\tilde{z}>-a/\rho$, completing the proof of the claim.

We next consider the case $g(-a,-a/\rho)\le0$. We claim that for
this case, $\{z_{-a,i}\}_{i\in[m]}$ is a singleton. We next prove
it. For this case, the equation \eqref{eq:geq} has at most two distinct
solutions in $(-a/\rho,(1-a)/\rho)$. Since adjacent solutions cannot
be both contained in $\{z_{-a,i}\}_{i\in[m]}$, only one solution
can be contained in $\{z_{-a,i}\}_{i\in[m]}$. In other words, $\{z_{-a,i}\}_{i\in[m]}\subseteq\{-a/\rho,\tilde{z}\}$
for some solution $\tilde{z}$. Suppose that $\{z_{-a,i}\}_{i\in[m]}=\{-a/\rho,\tilde{z}\}$
with $\tilde{z}>-a/\rho$. For this case, by the fact that $G$ is
an antiderivative of $g$, we have that $G(-a,-a/\rho)=G(-a,\tilde{z})$
only if $g(-a,z)=0$ has two distinct solutions in $(-a/\rho,(1-a)/\rho)$,
denoted by $-a/\rho<z_{1}<z_{2}\le(1-a)/\rho$, and the solution contained
in $\{z_{-a,i}\}_{i\in[m]}$ is $z_{2}$. Moreover, $G(-a,-a/\rho)=G(-a,z_{2})$
if and only if the two areas enclosed by the curve $z\mapsto\rho\Phi'(a+\rho z)$
and the lines $z\mapsto\lambda(2z+a)-\eta$ and $z=-a/\rho$ have
the same size. However, since $\Phi'$ is symmetric w.r.t. $1/2$,
the area at the left side of $z=z_{1}$ is strictly larger than the
one at the right side (note that $z_{2}$ is strictly smaller than
$(1-a)/\rho$ by Statement 3 of Proposition \ref{prop:property}).
This leads to a contradiction, and hence, $\{z_{-a,i}\}_{i\in[m]}$
is a singleton and moreover, $\{z_{-a,i}\}_{i\in[m]}=\{\tilde{z}\}$
with $\tilde{z}>-a/\rho$.

By symmetry, similar conclusions can be drawn for $s=1-a$. Summarizing
these points, optimal solutions to the maximization problem in \eqref{eq:generalboundsimple}
satisfy that $\{z_{-a,i}\}_{i\in[m]}\subseteq\{\tilde{z},b-\tilde{z}\}$,
$\{z_{1-a,i}\}_{i\in[m]}\subseteq\{\hat{z},b-\hat{z}\}$, and $\{z_{-a,i}\}_{i\in[m]}$
or $\{z_{1-a,i}\}_{i\in[m]}$ is a singleton, where $b=\frac{1-2a}{\rho}$,
$-a/\rho<\tilde{z}\le\frac{b}{2}\le\hat{z}<(1-a)/\rho$. Moreover,
$\hat{z}\ge\frac{b+1}{2}$ if $\{z_{-a,i}\}_{i\in[m]}$ has size $2$,
and $\tilde{z}\le\frac{b-1}{2}$ if $\{z_{1-a,i}\}_{i\in[m]}$ has
size $2$. We next divide the rest of the proof to three cases.

Case 1: We first consider the case in which both $\{z_{-a,i}\}_{i\in[m]}$
and $\{z_{1-a,i}\}_{i\in[m]}$ are singletons. Solving $\mathbb{E}Z=0,\mathbb{E}[SZ]=\mathbb{E}[Z^{2}]$
yields $Z=0$ or $Z=S$ with probability one. By Statement 3 of Proposition
\ref{prop:property}, $Z=0$ is not an optimal solution to the optimization
problem in \eqref{eq:generalboundsimple}. On the other hand, $Z=S$
results in the following value of the program: 
\[
\mathbb{E}[\Phi(a+\rho S)]=(1-a)\Phi(a+\rho(-a))+a\Phi(a+\rho(1-a)).
\]

Case 2: We next consider the case in which $\{z_{-a,i}\}_{i\in[m]}$
is of size $2$ and $\{z_{1-a,i}\}_{i\in[m]}$ is a singleton. In
this case, the optimal solution to the optimization problem in \eqref{eq:generalboundsimple}
is 
\begin{equation}
P_{SZ}(s,z)=\begin{cases}
1-a-p & s=-a,z=\tilde{z}\\
p & s=-a,z=b-\tilde{z}\\
a & s=1-a,z=\hat{z}
\end{cases}\label{eq:-10}
\end{equation}
for some $(p,\tilde{z},\hat{z})$ such that $0<p<1-a,-a/\rho<\tilde{z}\le\frac{b}{2}\le\hat{z}<(1-a)/\rho$,
and $\hat{z}\ge\frac{b+1}{2}$. Solving $\mathbb{E}Z=0,\mathbb{E}[SZ]=\mathbb{E}[Z^{2}]$
yields the feasible solution satisfying $\hat{z}\ge\frac{b+1}{2}$
which is given by 
\begin{align}
 & \hat{z}=\frac{1}{2}(\sqrt{\frac{\Delta}{a}}+b+1),\label{eq:-25}\\
 & p=\frac{-\sqrt{a\Delta}-(a(b+1-2\tilde{z})+2\tilde{z})}{2(b-2\tilde{z})}\label{eq:-26}
\end{align}
with $\Delta=a(b^{2}-4b\tilde{z}+2b+4\tilde{z}^{2}+1)+4\tilde{z}(b-\tilde{z})$.
To ensure $\Delta\ge0,p\ge0$, it is required that $\tilde{z}\le\frac{-a(1+b)}{2(1-a)}$.
Hence, the optimal value of the program for this case is $\hat{\Gamma}_{\rho}(a)$.

Case 3: We lastly consider the case in which $\{z_{1-a,i}\}_{i\in[m]}$
is of size $2$ and $\{z_{-a,i}\}_{i\in[m]}$ is a singleton. In this
case, the optimal solution to the optimization problem in \eqref{eq:generalboundsimple}
is 
\begin{equation}
P_{SZ}(s,z)=\begin{cases}
1-a & s=-a,z=\tilde{z}\\
p & s=1-a,z=b-\hat{z}\\
a-p & s=1-a,z=\hat{z}
\end{cases}\label{eq:-10-1}
\end{equation}
for some $(p,\tilde{z},\hat{z})$ such that $0<p<a,-a/\rho<\tilde{z}\le\frac{b}{2}\le\hat{z}<(1-a)/\rho$,
and $\tilde{z}\le\frac{b-1}{2}$. By symmetry, substituting $a\leftarrow1-a,\tilde{z}\leftarrow-\hat{z},\hat{z}\leftarrow-\tilde{z}$
(which implies $b\leftarrow-b$) into \eqref{eq:-25} and \eqref{eq:-26},
we obtain the optimal $p,\tilde{z}$ for this case, which results
in the value $\hat{\Gamma}_{\rho}(1-a)$ of the program. Hence, the
first part of Statement 2 holds.

For $a=1/2$, we have $b=0$ and both $\frac{-a(1+b)}{2(1-a)}$ and
$c$ are equal to $-1/2$. Hence, $\tilde{z}=-1/2$ is the unique
feasible solution to the supremization in \eqref{eq:-37}, which results
in the value $\Phi(\frac{1-\rho}{2})$. Furthermore, for $a=1/2$,
$\mathbb{E}[\Phi(a+\rho S)]$ is the same value. Hence, the second
part of Statement 2 follows.

\section{\label{sec:Proof-of-Theorem-CKConjecture-1}Proof of Corollary \ref{cor:CK_Conjecture}}

Excluding the trivial cases, we assume $\rho\in(0,\rho^{*}]$. By
definition, $\Phi=\Phi_{\alpha}^{\mathrm{sym}}$ with $\alpha\in(1,2)$
satisfies the assumption in Statement 1 of Theorem \ref{thm:symmetric},
which means that Statement 1 of Theorem \ref{thm:symmetric} can be
applied to this case. For $a=1/2$, we have $z_{1}+z_{2}=0$, which
further implies $p=q$. From $(z_{1},z_{2})\in\overline{\mathcal{Z}}$,
we know $-\frac{1}{2}\leq z_{1}\leq-\frac{1}{4}$. Denote the bound
in \eqref{eq:generalbound} for $\Phi=\Phi_{\alpha}^{\mathrm{sym}}$
as $\overline{\Gamma}_{\rho}^{(\alpha)}(a)$ where $\alpha\in[1,2)$.
Then, by the monotonicity of $\Phi_{\alpha}^{\mathrm{sym}}$ in $\alpha$,
we have for $\alpha\in(1,2)$, 
\begin{align*}
\overline{\Gamma}_{\rho}^{(1)}(\frac{1}{2}) & \leq\lim_{\alpha\downarrow1}\overline{\Gamma}_{\rho}^{(\alpha)}(\frac{1}{2})=\lim_{\alpha\downarrow1}\max_{-\frac{1}{2}\leq z_{1}\leq-\frac{1}{4}}h_{\alpha}(z_{1}),
\end{align*}
where 
\begin{align*}
h_{\alpha}(z_{1}) & :=2p\Phi_{\alpha}^{\mathrm{sym}}(\frac{1}{2}+\rho z_{1})\\
 & =\frac{1-\rho}{(1+2\rho z_{1})(1-\rho-2\rho z_{1})}\Phi_{\alpha}^{\mathrm{sym}}(\frac{1}{2}+\rho z_{1}).
\end{align*}

We now swap the limit and maximization. Denote $\{\alpha_{k}\}$ as
a decreasing sequence with limit $1$. Denote $z_{1,k}$ as an optimal
solution to the optimization $\max_{-\frac{1}{2}\leq z_{1}\leq-\frac{1}{4}}h_{\alpha}(z_{1})$
with $\alpha=\alpha_{k}$. By passing to a subsequence, we assume
that $\{z_{1,k}\}$ converges to some $z_{1}^{*}$. Then, by the continuity
of $h_{\alpha}(z_{1})$ in $(\alpha,z_{1})$, $\overline{\Gamma}_{\rho}^{(\alpha_{k})}(\frac{1}{2})=h_{\alpha_{k}}(z_{1,k})\to h_{1}(z_{1}^{*})\le\max_{-\frac{1}{2}\leq z_{1}\leq-\frac{1}{4}}h_{1}(z_{1})$,
which implies $\overline{\Gamma}_{\rho}^{(1)}(\frac{1}{2})\leq\max_{-\frac{1}{2}\leq z_{1}\leq-\frac{1}{4}}h(z_{1}),$
with $h(z_{1}):=h_{1}(z_{1})$ for brevity. More explicitly, 
\begin{align*}
h(z_{1}) & =\frac{1-\rho}{(1+2\rho z_{1})(1-\rho-2\rho z_{1})}((\frac{1}{2}+\rho z_{1})\log(\frac{1}{2}+\rho z_{1})+(\frac{1}{2}-\rho z_{1})\log(\frac{1}{2}-\rho z_{1})).
\end{align*}
To compute the maximum of $h(z_{1})$ over $-\frac{1}{2}\leq z_{1}\leq-\frac{1}{4}$,
we take its derivative which is a standard technique. That is, 
\begin{equation}
h'(z_{1})=\frac{-\rho(1-\rho)\varphi_{\rho}(z_{1})}{(1+2\rho z_{1}){}^{2}(1-\rho-2\rho z_{1}){}^{2}},\label{eq:-40}
\end{equation}
where 
\[
\varphi_{\rho}(z_{1}):=((1-2\rho z_{1})^{2}-2\rho)\log(\frac{1}{2}-\rho z_{1})-(1+2\rho z_{1})^{2}\log(\frac{1}{2}+\rho z_{1}).
\]

We claim that $h$ is non-increasing given $\rho\in(0,\rho^{*}]$.
We next prove it. It is easy to verify that 
\[
\varphi_{\rho}''(z_{1})=\frac{8\rho^{2}(\rho+(1-2\rho z_{1}){}^{2}\log(\frac{1-2\rho z_{1}}{1+2\rho z_{1}}))}{(1-2\rho z_{1}){}^{2}}\ge0.
\]
Hence, $\varphi'$ is increasing. Furthermore, $\varphi_{\rho}'(-\frac{1}{2})=\frac{4\rho g(\rho)}{\rho+1},$
where 
\[
g(\rho):=-(1-\rho^{2})\log(\frac{1-\rho}{2})-(1+\rho)^{2}\log(\frac{1+\rho}{2})-1.
\]
It is easy to verify that $g''(\rho)\leq0$, and $g(0),g(\frac{1}{2})\geq0$.
Hence, $g(\rho)\geq0$ (i.e., $\varphi_{\rho}'(-\frac{1}{2})\geq0$)
for $\rho\in(0,\frac{1}{2}]$. Combining this with the fact that $\varphi_{\rho}'$
is increasing gives $\varphi_{\rho}'(z_{1})\geq0$ for $z_{1}\in[-\frac{1}{2},0]$
and $\rho\in(0,\frac{1}{2}]$. Hence, for $\rho\in(0,\frac{1}{2}]$,
$\varphi_{\rho}$ is increasing on $[-\frac{1}{2},0]$.

Recall $\psi(\rho)$ defined in \eqref{eq:psi}, and observe that
$\varphi_{\rho}(-\frac{1}{2})=\psi(\rho)$. We now prove $\psi(\rho)\ge0$
for $\rho\in[0,\rho^{*}]$. It is easy to verify $\psi'''(\rho)\ge0$,
and hence $\psi'(\rho)$ is convex. Also, $\psi'(0)\geq0,\psi'(\frac{1}{2})\leq0$.
Hence, $\psi'$ is first-positive-then-negative\footnote{We say a function $f$ is \emph{first-positive-then-negative} if there
exists some real number $a$ such that $f(x)\ge0$ for $x<a$ and
$f(x)\le0$ for $x>a$. We can define a function to be \emph{first-negative-then-positive},
\emph{first-increasing-then-decreasing}, and \emph{first-decreasing-then-increasing}
in a similar way.} on $[0,\frac{1}{2}]$, which means that $\psi$ is first-increasing-then-decreasing
on $[0,\frac{1}{2}]$. Observe that $\psi(0),\psi(\rho^{*})=0$ and
$\rho^{*}\in[0,\frac{1}{2}]$. Hence $\psi(\rho)\ge0$ for $\rho\in(0,\rho^{*}]$,
i.e., $\varphi_{\rho}(-\frac{1}{2})\ge0$.

Combining the facts that $\varphi_{\rho}$ is increasing and $\varphi_{\rho}(-\frac{1}{2})\ge0$
for $\rho\in(0,\rho^{*}]$ yields that $\varphi_{\rho}(z_{1})\ge0$
for $z_{1}\in[-\frac{1}{2},0]$ and $\rho\in(0,\rho^{*}]$. It means
that $h'(z_{1})\leq0$ and hence $h$ is non-increasing given $\rho\in(0,\rho^{*}]$,
completing the proof of the claim above.

Therefore, by the claim above,
\[
\overline{\Gamma}_{\rho}^{(1)}(\frac{1}{2})\le h(-\frac{1}{2})=(\frac{1-\rho}{2})\log(\frac{1-\rho}{2})+(\frac{1+\rho}{2})\log(\frac{1+\rho}{2}).
\]
Observe that the most RHS above is attained by dictator functions,
which completes the proof.

\section{\label{sec:Proof-of-Theorem-symmetricpower}Proof of Corollary \ref{cor:symmetricpower}}

Excluding the trivial cases, we assume $\rho\in(0,1)$. We next prove
Corollary \ref{cor:symmetricpower}. The case $\alpha\in(1,2)$ is
implied immediately by Lemma \ref{lem:BarnesOzgur} and Corollary
\ref{cor:CK_Conjecture}. For $\alpha\in[2,5]$, by Lemma \ref{lem:BarnesOzgur},
it suffices to only prove the case $\alpha=5$.

For $a=1/2$ and $\Phi(t)=t^{\alpha}+(1-t)^{\alpha}$ with $\alpha=5$,
similarly as in Section \ref{eq:psi}, by applying Statement 1 of
Theorem \ref{thm:symmetric}, we have $z_{1}+z_{2}=0$, and hence
$p=q$, $-\frac{1}{2}\leq z_{1}\leq-\frac{1}{4}$, and 
\[
\overline{\Gamma}_{\rho}(\frac{1}{2})=\sup_{-\frac{1}{2}\leq z_{1}\leq-\frac{1}{4}}2[\frac{1}{2}-p+p\Phi(\frac{1}{2}+\rho z_{1})]
\]
with $p=\frac{1-\rho}{2(1-\rho-2\rho z_{1})(1+2\rho z_{1})}.$ Denoting
$z=-2\rho z_{1}\in[\frac{\rho}{2},\rho]$, we have $\overline{\Gamma}_{\rho}(\frac{1}{2})=\sup_{\frac{\rho}{2}\le z\leq\rho}h(z)$,
where 
\[
h(z):=\frac{-\rho+5(\rho-1)z^{3}+5(\rho-1)z^{2}+15\rho z+z+1}{16(-\rho+z+1)}.
\]
By taking derivatives and using the standard arguments, it can be
observed that $h(z)$ is non-decreasing. Hence $\overline{\Gamma}_{\rho}(a)=h(\rho)=\Phi(\frac{1-\rho}{2})$,
completing the proof. We next provide the details on how to take derivatives
to show that $h(z)$ is non-decreasing. 

Taking derivative for $h$, we have $h'(z)=\frac{5\varphi(z)}{16(-\rho+z+1)^{2}},$
where 
\[
\varphi(z)=-3(\rho-1)\rho+2(\rho-1)z^{3}+(-3\rho^{2}+7\rho-4)z^{2}-2(\rho-1)^{2}z.
\]
Taking derivative again, we further have $\varphi'(z)=-2(1-\rho)(3z+1)(1-\rho+z)\leq0.$
Hence, $\varphi(z)$ is non-increasing, which implies that $\varphi(z)\ge\varphi(\rho)=(1-\rho)^{3}\rho\ge0$,
i.e., $h'(z)\ge0$. Therefore, $h(z)$ is non-decreasing. 

\section{\label{sec:Proof-of-Theorem-asymmetricpower}Proof of Corollary \ref{cor:asymmetricpower}}

Excluding the trivial cases, we assume $\rho\in(0,1)$.

\subsection{Case of $\alpha\in(1,2)$}

Here, we denote $\Phi(t)=t^{\alpha}$ with $\alpha\in(1,2)$. Observe
that $\Phi'$ is concave. By Theorem \ref{thm:derivativeconcave},
for $a=1/2$, 
\[
\widetilde{\Gamma}_{\rho}(\frac{1}{2})=\sup_{0\le z_{2}\le\frac{1}{2}}p(\frac{1}{2}+\rho z_{1})^{\alpha}+\frac{1}{2}(\frac{1}{2}+\rho z_{2})^{\alpha}.
\]
where $p,z_{2}$ are defined in \eqref{eq:z1} and \eqref{eq:p}.
Now we consider $p$ as a free variable, and solving \eqref{eq:z1}
and \eqref{eq:p} w.r.t. $z_{1},z_{2}$, we obtain two solutions:
\begin{align}
 & \begin{cases}
z_{1}^{(1)}=\frac{p(1-2p-\rho)-T(p)}{2(1+2p)p\rho},\\
z_{2}^{(1)}=\frac{\frac{1}{2}-p(1-\rho)+T(p)}{(1+2p)\rho},
\end{cases}\label{eq:-59}
\end{align}
and 
\begin{align}
 & \begin{cases}
z_{1}^{(2)}=\frac{p(1-2p-\rho)+T(p)}{2(1+2p)p\rho},\\
z_{2}^{(2)}=\frac{\frac{1}{2}-p(1-\rho)-T(p)}{(1+2p)\rho},
\end{cases}\label{eq:-60}
\end{align}
where 
\begin{equation}
T(p)=\sqrt{p(p(2-2\rho+\rho^{2})+\rho-1)}.\label{eq:-63}
\end{equation}
To ensure the term in the square root above is nonnegative, it is
required that 
\begin{equation}
\frac{1-\rho}{2-2\rho+\rho^{2}}\le p\le\frac{1}{2}.\label{eq:-61}
\end{equation}
Observe that given $p$, $pz_{1}+\frac{1}{2}z_{2}$ remains the same
for $(z_{1}^{(1)},z_{2}^{(1)})$ and $(z_{1}^{(2)},z_{2}^{(2)})$,
and moreover, $z_{1}^{(1)}\le z_{1}^{(2)}\le z_{2}^{(2)}\le z_{2}^{(1)}$.
By the convexity of $\Phi$, we have that $(z_{1}^{(1)},z_{2}^{(1)})$
leads to a larger value of $p(\frac{1}{2}+\rho z_{1})^{\alpha}+\frac{1}{2}(\frac{1}{2}+\rho z_{2})^{\alpha}$
than $(z_{1}^{(2)},z_{2}^{(2)})$. Hence, 
\[
\widetilde{\Gamma}_{\rho}(\frac{1}{2})\leq\sup_{\frac{1-\rho}{2-2\rho+\rho^{2}}\le p\le\frac{1}{2}}h_{\alpha}(p)
\]
where

\begin{align*}
h_{\alpha}(p) & :=p(\frac{1}{2}+\rho z_{1}^{(1)})^{\alpha}+\frac{1}{2}(\frac{1}{2}+\rho z_{2}^{(1)})^{\alpha}\\
 & =p(\frac{(2-\rho)p-T(p)}{2(1+2p)p})^{\alpha}+\frac{1}{2}(\frac{1+p\rho+T(p)}{1+2p})^{\alpha}.
\end{align*}

Define 
\begin{align}
\varphi(\alpha,p) & :=-A(p)D(p)^{\alpha-1}-\frac{B(p)}{\alpha}+C(p)\label{eq:phi}
\end{align}
where 
\begin{align*}
A(p) & :=p[(-1+6p)(1-\rho)+2p\rho^{2}-2(2-\rho)T(p)]\\
B(p) & :=[(2-\rho)p-T(p)]2(1+2p)T(p)\\
C(p) & :=p[(3-2p-8p^{2})(1-\rho)-4\rho^{2}p^{2}+4p(2-\rho)T(p)]\\
D(p) & :=\frac{(p\rho+1+T(p))2p}{(2-\rho)p-T(p)}.
\end{align*}
It is easy to verify that $A(p),B(p),D(p)\ge0$ for $\frac{1-\rho}{2-2\rho+\rho^{2}}\le p\le\frac{1}{2}$.
Then, 
\begin{equation}
\frac{\mathrm{d}}{\mathrm{d}p}h_{\alpha}(p)=(\frac{(2-\rho)p-T(p)}{2(1+2p)p})^{\alpha-1}\frac{-\alpha\varphi(\alpha,p)}{4(1+2p)^{2}pT(p)}.\label{eq:-38}
\end{equation}

Recall that the function $\theta$ is defined in \eqref{eq:theta-1}.
Obviously, it holds that $\theta^{-1}(\frac{1-\rho}{1+\rho})\ge1$
for $0\le\rho\le1$. 
\begin{lem}
\label{lem:phi}For $\theta^{-1}(\frac{1-\rho}{1+\rho})\le\alpha<2$
and $\frac{1-\rho}{2-2\rho+\rho^{2}}\le p\le\frac{1}{2}$, we have
$\varphi(\alpha,p)\leq0$. 
\end{lem}

By this lemma and \eqref{eq:-38}, for $\theta^{-1}(\frac{1-\rho}{1+\rho})\le\alpha<2$,
$h_{\alpha}$ is non-decreasing for $\frac{1-\rho}{2-2\rho+\rho^{2}}\le p\le\frac{1}{2}$,
which further implies that $\widetilde{\Gamma}_{\rho}(\frac{1}{2})\leq h_{\alpha}(\frac{1}{2})$.
Observe that $h_{\alpha}(\frac{1}{2})=\frac{1}{2}\Phi(\frac{1-\rho}{2})+\frac{1}{2}\Phi(\frac{1+\rho}{2})$
is the $\Phi$-stability of dictator functions. Hence, $\widetilde{\Gamma}_{\rho}(\frac{1}{2})$
is attained by dictator functions for $\theta^{-1}(\frac{1-\rho}{1+\rho})\le\alpha<2$
(or equivalently, for $\rho\leq\frac{1-\theta(\alpha)}{1+\theta(\alpha)}$),
which completes the proof of Corollary \ref{cor:asymmetricpower}
for the case $\alpha\in(1,2)$. Hence, the rest is to prove the lemma
above. 
\begin{proof}[Proof of Lemma \ref{lem:phi}]
Since $A(p),B(p),D(p)\ge0$, by definition, given $p$, $\varphi(\alpha,p)$
is concave in $\alpha$. It is easy to verify that $\varphi(1,p)=0$.
Hence, given $p$, one of the following two statements is true: 1)
$\varphi(\alpha,p)\le0$ for any $\alpha>0$; 2) the equation $\varphi(\alpha,p)=0$
with unknown $\alpha$ has exactly two distinct solutions, $1$ and
$\alpha^{*}(p)>0$. If the first statement is true, then Lemma \ref{lem:phi}
holds for such $p$. Hence, it suffices to consider the case in which
the second statement is true.

If $p$ satisfies that $\alpha^{*}(p)<1$, by the concavity of $\varphi(\alpha,p)$
in $\alpha$, we have that $\varphi(\alpha,p)\le0$ for any $\alpha\ge1$.
Hence, Lemma \ref{lem:phi} holds for such $p$.

We next consider the case $\alpha^{*}(p)>1$. Similarly, by the concavity
of $\varphi(\alpha,p)$ in $\alpha$, we have that $\varphi(\alpha,p)\le0$
for any $\alpha\ge\alpha^{*}(p)$. Hence, Lemma \ref{lem:phi} holds
for $\alpha\ge\sup_{\frac{1-\rho}{2-2\rho+\rho^{2}}\le p\le\frac{1}{2}}\alpha^{*}(p)$.
To determine the supremization here, we need the following technical
lemma. Its proof is deferred to Section \ref{subsec:Proof-of-Lemma}. 
\begin{lem}
\label{lem:-1} For any $\frac{1-\rho}{2-2\rho+\rho^{2}}\le p\le\frac{1}{2}$
such that $1<\alpha^{*}(p)<2$, we have $\frac{\mathrm{d}}{\mathrm{d}p}\alpha^{*}(p)\ge0$. 
\end{lem}

By the lemma above, $\sup_{\frac{1-\rho}{2-2\rho+\rho^{2}}\le p\le\frac{1}{2}}\alpha^{*}(p)=\alpha^{*}(\frac{1}{2})$.
It is easy to check that $\alpha^{*}(\frac{1}{2})=\theta^{-1}(\frac{1-\rho}{1+\rho})$.
Hence Lemma \ref{lem:phi} holds for $\alpha\ge\theta^{-1}(\frac{1-\rho}{1+\rho})$.

Combining all the cases above completes the proof of Lemma \ref{lem:phi}. 
\end{proof}

\subsubsection{\label{subsec:Proof-of-Lemma}Proof of Lemma \ref{lem:-1}}

Denote $\varphi_{1}(\alpha,p):=\frac{\partial}{\partial\alpha}\varphi(\alpha,p)$
and $\varphi_{2}(\alpha,p):=\frac{\partial}{\partial p}\varphi(\alpha,p)$.
By the assumption $\alpha^{*}(p)>1$, the definition of $\alpha^{*}(p)$,
and the concavity of $\varphi(\alpha,p)$ in $\alpha$, we have $\varphi_{1}(\alpha^{*}(p),p)<0$.
By the implicit function theorem, 
\[
\frac{\mathrm{d}}{\mathrm{d}p}\alpha^{*}(p)=\frac{-\varphi_{2}(\alpha^{*}(p),p)}{\varphi_{1}(\alpha^{*}(p),p)}.
\]
Hence, it suffices to show $\varphi_{2}(\alpha^{*}(p),p)\ge0$.

Observe that 
\begin{align}
\varphi_{2}(\alpha,p) & =-A'(p)D(p)^{\alpha-1}-(\alpha-1)A(p)D(p)^{\alpha-2}D'(p)-\frac{B'(p)}{\alpha}+C'(p)\label{eq:-86}
\end{align}
Since $\alpha^{*}=\alpha^{*}(p)$ is a solution, i.e., $\varphi(\alpha^{*},p)=0$,
we obtain $A(p)D(p)^{\alpha^{*}-1}=C(p)-\frac{B(p)}{\alpha^{*}}$.
Substituting it into \eqref{eq:-86} with $\alpha$ replaced by $\alpha^{*}$,
we obtain 
\begin{align}
\varphi_{2}(\alpha^{*},p) & =-A'(p)D(p)^{\alpha^{*}-1}-(\alpha^{*}-1)A(p)D(p)^{\alpha^{*}-2}D'(p)-\frac{B'(p)}{\alpha^{*}}+C'(p)\nonumber \\
 & =-A(p)D(p)^{\alpha^{*}-1}\Big(\frac{A'(p)}{A(p)}+(\alpha^{*}-1)\frac{D'(p)}{D(p)}\Big)-\frac{B'(p)}{\alpha^{*}}+C'(p)\nonumber \\
 & =-(C(p)-\frac{B(p)}{\alpha^{*}})\Big(\frac{A'(p)}{A(p)}+(\alpha^{*}-1)\frac{D'(p)}{D(p)}\Big)-\frac{B'(p)}{\alpha^{*}}+C'(p)\nonumber \\
 & =\phi(\alpha^{*},p),\label{eq:-62}
\end{align}
where 
\begin{align}
\phi(\alpha,p)& :=-C(p)\Big(\frac{A'(p)}{A(p)}+(\alpha-1)\frac{D'(p)}{D(p)}\Big)+\frac{B(p)}{\alpha}\Big(\frac{A'(p)}{A(p)}-\frac{D'(p)}{D(p)}-\frac{B'(p)}{B(p)}\Big) \\
& \qquad +B(p)\frac{D'(p)}{D(p)}+C'(p).
\end{align}
By standard arguments, it can be shown that 
\begin{equation}
\frac{A'(p)}{A(p)}-\frac{D'(p)}{D(p)}-\frac{B'(p)}{B(p)}\le0;\label{eq:}
\end{equation}
see details below. Hence, given $\frac{1-\rho}{2-2\rho+\rho^{2}}\le p\le\frac{1}{2}$,
$\phi(\alpha,p)$ is concave in $\alpha>0$. 

On the other hand, using Mathematica, it is easy to show that $\phi(1,p)=0,\phi(2,p)\ge0$
for $\frac{1-\rho}{2-2\rho+\rho^{2}}\le p\le\frac{1}{2}$. Hence,
$\phi(\alpha,p)\ge0$ for $1<\alpha<2$ and $\frac{1-\rho}{2-2\rho+\rho^{2}}\le p\le\frac{1}{2}$,
which, combined with \eqref{eq:-62}, implies $\varphi_{2}(\alpha^{*}(p),p)\ge0$
since by assumption, $1<\alpha^{*}(p)<2$. This completes the proof
of Lemma \ref{lem:-1}.

For completeness, we next provide proof details of \eqref{eq:}. Observe
that

\[
\frac{A'(p)}{A(p)}-\frac{D'(p)}{D(p)}-\frac{B'(p)}{B(p)}=\frac{p(F(p)+E(p)T(p))}{2(2p+1)T(p)^{2}(p\rho+1+T(p))A(p)},
\]
where 
\begin{align*}
E(p) & :=-16(2\rho^{4}-9\rho^{3}+17\rho^{2}-16\rho+6)p^{3}\\
 & \qquad-8(\rho^{4}+\rho^{3}-10\rho^{2}+14\rho-7)p^{2}-8(\rho^{3}-\rho^{2}-2\rho+2)p-2(\rho-1)^{2},
\end{align*}
and 
\begin{align*}
F(p):= & -16(2\rho^{5}-11\rho^{4}+27\rho^{3}-36\rho^{2}+26\rho-8)p^{4} \\
& \qquad -4(2\rho^{5}+4\rho^{4}-39\rho^{3}+86\rho^{2}-83\rho+32)p^{3}\\
 & \qquad -4(3\rho^{4}-4\rho^{3}-5\rho^{2}+13\rho-7)p^{2}-(\rho-1)^{2}(5\rho+4)p-(\rho-1)^{2}.
\end{align*}
Using Mathematica\footnote{The code can be found in the link https://www.dropbox.com/s/kll3157wfuwaw7c/code.nb?dl=0},
one can show that $F(p)\le0$ and $F(p)^{2}\geq E(p)^{2}T(p)^{2}$
for $\frac{1-\rho}{2-2\rho+\rho^{2}}\le p\le\frac{1}{2}$ (and $0<\rho<1$).
Hence, $F(p)+E(p)T(p)\le0$, which in turn implies \eqref{eq:}. 

\subsection{Case of $\alpha\in[2,3]$}

For $\alpha\in[2,3]$, by Lemma \ref{lem:BarnesOzgur}, to prove Corollary
\ref{cor:asymmetricpower}, it suffices to prove the case $\alpha=3$.
Here we let $\Phi(t)=t^{3}$. Observe that $\Phi'$ is convex. By
Remark \ref{rem:convex}, we have $\widetilde{\Gamma}_{\rho}(a)=\sup_{0\le z_{2}\leq a}g(z_{2}),$
where 
\[
g(z_{2}):=a-p+p(a-\rho z_{1})^{3}+(1-a)(a-\rho z_{2})^{3}
\]
with 
\begin{align*}
 & z_{1}=\frac{z_{2}(\rho(1-z_{2})-(1-a))}{a-\rho z_{2}},\\
 & p=\frac{(1-a)(a-\rho z_{2}){}^{2}}{1-a+\rho^{2}z_{2}-(1-a+\rho z_{2})^{2}}.
\end{align*}
For $a=1/2$, by defining $z=2\rho z_{2}\in[0,\rho]$, we have $\widetilde{\Gamma}_{\rho}(a)=\sup_{0\le z\leq\rho}h(z)$,
where 
\[
h(z):=g(\frac{z}{2\rho})=\frac{1-(1-4\rho)z-(2\rho^{2}+2\rho+1)z^{2}-(1-2\rho)z^{3}}{8(1-z)}.
\]

By standard arguments, it can be shown that $h'(z)\ge0$; see details
below. Hence, $h(z)$ is increasing, which implies $\widetilde{\Gamma}_{\rho}(a)=h(\rho)=\frac{1}{2}\Phi(\frac{1+\rho}{2})+\frac{1}{2}\Phi(\frac{1-\rho}{2})$.
Note that the most RHS is attained by dictator functions, completing
the proof.

For completeness, we next provide proof details of $h'(z)\ge0$. Taking
derivative, we have $h'(z)=\frac{\varphi(z)}{4(z-1)^{2}},$ where
\[
\varphi(z):=(1-2\rho)z^{3}+(\rho^{2}+4\rho-1)z^{2}-(2\rho^{2}+2\rho+1)z+2\rho.
\]
Taking derivative again, we have 
\[
\varphi'(z)=(3-6\rho)z^{2}+2(\rho^{2}+4\rho-1)z-2\rho^{2}-2\rho-1.
\]
For $\rho\le1/2$, $\varphi'(z)$ is convex in $z$. Hence for this
case, $\varphi'(z)\le\max\{\varphi'(0),\varphi'(\rho)\}\le0$. For
$\rho>1/2$, $\varphi'(z)$ is increasing in $z$ for $z\leq\rho\leq\frac{\rho^{2}+4\rho-1}{6\rho-3}$.
Hence for this case, $\varphi'(z)\le\varphi'(\rho)\le0$. Combining
these two cases, $\varphi'(z)\le0$ and hence $\varphi(z)$ is decreasing
for $\rho\in[0,1]$. Hence $\varphi(z)\ge\varphi(\rho)=(1-\rho)^{3}\rho\ge0$,
which implies $h'(z)\ge0$. 

\section{\label{sec:Proof-of-Lemma}Proof of Lemma \ref{lem:W-W}}

Let $\mathcal{B}$ be the support of $f$. Hence $|\mathcal{B}|=a2^{n}$.
Denote $\mathcal{A}_{1}:=\{\mathbf{x}:x_{1}=1\}$ and $\mathcal{A}_{1}^{c}:=\{-1,1\}^{n}\backslash\mathcal{A}_{1}=\{\mathbf{x}:x_{1}=-1\}$.
Then by definition, $|\mathcal{B}\cap\mathcal{A}_{1}|=\frac{a+\beta}{2}2^{n}$.
Let $\mathcal{A}\subseteq\mathcal{A}_{1}$ be an arbitrary subset
such that $\mathcal{B}\cap\mathcal{A}_{1}\subseteq\mathcal{A}$ and
$|\mathcal{A}|=a2^{n}$. Denote $g:=1_{\mathcal{A}},h:=1\{\mathbf{x}\in\{-1,1\}^{n-1}:(1,\mathbf{x})\in\mathcal{A},(1,\mathbf{x})\notin\mathcal{B}\}$,
and $l:=1\{\mathbf{x}\in\{-1,1\}^{n-1}:(-1,\mathbf{x})\in\mathcal{B}\}$.
Then $\mathbb{E}h=\mathbb{E}l=\frac{a-\beta}{2}2^{n}$. Moreover,
$f=g-h\cdot1\{x_{1}=1\}+l\cdot1\{x_{1}=-1\}$, and hence, $\hat{f}_{\mathcal{S}}=\hat{g}_{\mathcal{S}}-\frac{1}{2}\hat{h}_{\mathcal{S}\backslash\{1\}}+\frac{1}{2}\hat{l}_{\mathcal{S}\backslash\{1\}},\forall\mathcal{S}\subseteq[n]$.
Note that $h$ and $l$ are Boolean functions on the $(n-1)$-dimensional
space. Hence, their Fourier coefficients $\hat{h}_{\mathcal{S}},\hat{l}_{\mathcal{S}},\mathcal{S}\subseteq[2:n]$
are also defined on the $(n-1)$-dimensional space. By the Minkowski
inequality, 
\begin{align}
\mathbf{W}_{1}[f]-\beta^{2} & =\sum_{i=2}^{n}\hat{f}_{\{i\}}^{2}=\sum_{i=2}^{n}(\hat{g}_{\{i\}}-\frac{1}{2}\hat{h}_{\{i\}}+\frac{1}{2}\hat{l}_{\{i\}})^{2}\nonumber \\
 & \leq\Big(\sqrt{\sum_{i=2}^{n}\hat{g}_{\{i\}}^{2}}+\sqrt{\frac{1}{4}\sum_{i=2}^{n}\hat{h}_{\{i\}}^{2}}+\sqrt{\frac{1}{4}\sum_{i=2}^{n}\hat{l}_{\{i\}}^{2}}\Big)^{2}\nonumber \\
 & =\Big(\sqrt{\mathbf{W}_{1}[g]-a^{2}}+\frac{1}{2}\sqrt{\mathbf{W}_{1}[h]}+\frac{1}{2}\sqrt{\mathbf{W}_{1}[l]}\Big)^{2}\label{eq:-66}\\
 & \leq\Big(\sqrt{W^{(n)}(a)-a^{2}}+\sqrt{W^{(n-1)}(a-\beta)}\Big)^{2}\nonumber 
\end{align}
where \eqref{eq:-66} follows by the facts that $\hat{g}_{\{1\}}=a$
and $\mathbf{W}_{1}[h]=\sum_{i=2}^{n}\hat{h}_{\{i\}}^{2},\mathbf{W}_{1}[l]=\sum_{i=2}^{n}\hat{l}_{\{i\}}^{2}$
since the Fourier coefficients of $h,l$ are defined on the $(n-1)$-dimensional
space. Therefore, \eqref{eq:-58} holds.

By Parseval's theorem, $W^{(n)}(a)\le a-a^{2}$. Substituting this
into \eqref{eq:-58} and specializing \eqref{eq:-58} to the case
$a=\frac{1}{2}$, we obtain \eqref{eq:-20-4}. On the other hand,
this upper bound is attained by the Boolean function 
\[
f(\mathbf{x})=1\{\mathbf{x}:x_{1}=1,(x_{2},x_{3},...,x_{n})\in\mathcal{A}^{c}\}+1\{-\mathbf{x}:x_{1}=1,(x_{2},x_{3},...,x_{n})\in\mathcal{A}\}
\]
for some $\mathcal{A}\subseteq\{-1,1\}^{n-1}$ such that $1_{\mathcal{A}}$
attains $W^{(n-1)}(\frac{1}{2}-k)$. This can be seen from the fact
that 
\[
f(\mathbf{x})=\frac{1+x_{1}}{2}-\sum_{\mathcal{S}\subseteq[2:n]:|\mathcal{S}|\textrm{ odd}}\hat{g}_{\mathcal{S}}\chi_{\mathcal{S}}(x_{2},..,x_{n})-x_{1}\sum_{\mathcal{S}\subseteq[2:n]:|\mathcal{S}|\textrm{ even}}\hat{g}_{\mathcal{S}}\chi_{\mathcal{S}}(x_{2},..,x_{n})
\]
where $\{\hat{g}_{\mathcal{S}}\}$ are Fourier coefficients of $g(x_{2},..,x_{n})=1_{\mathcal{A}}(x_{2},..,x_{n})$.

\section{\label{sec:Proof-of-Proposition-1}Proof of Proposition \ref{prop:property-1}}

The proofs of Statements 1-3 are exactly the same as those of Statements
1-3 of Proposition \ref{prop:property}. Here we only provide the
proof idea for Statement 4, and omit the details since it is similar
to that of Statement 4 of Proposition \ref{prop:property}. For even
$m\ge8$, denote $(z_{s,x,i},p_{s,x,i})_{(s,x,i)\in\mathcal{S}\times\mathcal{X}\times[m]}$
as an optimal solution to the maximization in \eqref{eq:generalboundsimple}
such that $z_{-s,-x,i}=-z_{s,x,i},p_{-s,-x,i}=p_{s,x,i}$ for all
$(s,x,i)$. Such an optimal solution always exists since for any optimal
solution $(\hat{z}_{s,x,i},\hat{p}_{s,x,i})_{(s,x,i)\in\mathcal{S}\times\mathcal{X}\times[m/2]}$
to the optimization problem in \eqref{eq:generalboundsimple} with
$m\leftarrow m/2$, we can construct a desired distribution $P_{SXZ}=\sum_{s,x,i}\frac{1}{2}\hat{p}_{s,x,i}(\delta_{s,x,\hat{z}_{s,x,i}}+\delta_{-s,-x,-\hat{z}_{s,x,i}})$,
which is optimal for the optimization in \eqref{eq:generalboundsimple-2}
and meanwhile satisfies $P_{SXZ}(-s,-x,-z)=P_{SXZ}(s,x,z)$ for all
$(s,x,z)$.

Suppose that LICQ is not satisfied by the solution $(z_{s,x,i},p_{s,x,i})_{(s,x,i)\in\mathcal{S}\times\mathcal{X}\times[m]}$.
Denote 
\begin{equation}
\mathcal{I}:=\{(s,x,i):a+\rho z_{s,x,i}=0\}\textrm{ and }\mathcal{J}:=\{(s,x,i):a+\rho z_{s,x,i}=1\}.\label{eq:IJ-1-1}
\end{equation}
By Statement 3, we can write $\mathcal{I}=\mathcal{I}^{-}\cup\mathcal{I}^{+}$
and $\mathcal{J}=\mathcal{J}^{-}\cup\mathcal{J}^{+}$ for some $\mathcal{I}^{-}\subseteq\{(-a,-1)\}\times[m]$,
$\mathcal{I}^{+}\subseteq\{(-a,1)\}\times[m]$, $\mathcal{J}^{-}\subseteq\{(1-a,-1)\}\times[m]$,
and $\mathcal{J}^{+}\subseteq\{(1-a,1)\}\times[m]$. The gradients
of the active inequality constraints (including \eqref{eq:con3-3}
by Statement 2) and the gradients of the equality constraints constitute
the $(|\mathcal{I}|+|\mathcal{J}|+5)\times(4m)$ matrix 
\begin{equation}
G:=\begin{bmatrix}\rho\mathbf{I}_{\mathcal{I}} & \mathbf{0}_{1\times2m}\\
-\rho\mathbf{I}_{\mathcal{J}} & \mathbf{0}_{1\times2m}\\
\mathbf{0}_{1\times2m} & (\mathbf{I}_{1\times m},\mathbf{0}_{1\times m})\\
\mathbf{0}_{1\times2m} & (\mathbf{0}_{1\times m},\mathbf{I}_{1\times m})\\
(p_{s,x,i})_{s,x,i} & (z_{s,x,i})_{s,x,i}\\
(xp_{s,x,i})_{s,x,i} & (xz_{s,x,i})_{s,x,i}\\
(p_{s,x,i}(2z_{s,x,i}-\rho s))_{s,x,i} & (z_{s,x,i}^{2}-\rho sz_{s,x,i})_{s,x,i}
\end{bmatrix}.\label{eq:G-3-1-1}
\end{equation}
The assumption that LICQ is not satisfied implies that $G$ is not
of full rank, or equivalently,

\[
\hat{G}:=\begin{bmatrix}(p_{s,-1,i})_{(s,i)\in(\mathcal{I}^{-}\cup\mathcal{J}^{-})^{c}} & \mathbf{0}_{1\times(2m-|\mathcal{I}^{+}|-|\mathcal{J}^{+}|)} & A^{-} & \mathbf{0}_{1\times2(m-1)}\\
\mathbf{0}_{1\times(2m-|\mathcal{I}^{-}|-|\mathcal{J}^{-}|)} & (p_{s,1,i})_{(s,i)\in(\mathcal{I}^{+}\cup\mathcal{J}^{+})^{c}} & \mathbf{0}_{1\times2(m-1)} & A^{+}\\
(p_{s,-1,i}(2z_{s,-1,i}-\rho s))_{(s,i)\in(\mathcal{I}^{-}\cup\mathcal{J}^{-})^{c}} & (p_{s,1,i}(2z_{s,1,i}-\rho s))_{(s,i)\in(\mathcal{I}^{+}\cup\mathcal{J}^{+})^{c}} & B^{-} & B^{+}
\end{bmatrix}
\]
is not of full rank where 
\begin{align*}
A^{-} & :=(z_{s,-1,i}-z_{s,-1,1})_{(s,i)\in\mathcal{S}\times[2:m]},\\
A^{+} & :=(z_{s,1,i}-z_{s,1,1})_{(s,i)\in\mathcal{S}\times[2:m]},\\
B^{-} & :=((z_{s,-1,i}-z_{s,-1,1})(z_{s,-1,i}+z_{s,-1,1}-\rho s))_{(s,i)\in\mathcal{S}\times[2:m]},\\
B^{+} & :=((z_{s,1,i}-z_{s,1,1})(z_{s,1,i}+z_{s,1,1}-\rho s))_{(s,i)\in\mathcal{S}\times[2:m]}.
\end{align*}
We next prove that under this condition, there is an optimal solution
of the form in \eqref{eq:-50}.

We first prove that for each $(s,x)$, $z_{s,x,i}$'s are identical
for all $i\in[m]$. Since the submatrix 
\[
\begin{bmatrix}(p_{-\frac{1}{2},-1,i})_{(-\frac{1}{2},i)\in(\mathcal{I}^{-}\cup\mathcal{J}^{-})^{c}} & (z_{-\frac{1}{2},-1,i}-z_{-\frac{1}{2},-1,1})_{i\in[2:m]}\\
(p_{-\frac{1}{2},-1,i}(2z_{-\frac{1}{2},-1,i}+\rho\frac{1}{2}))_{(-\frac{1}{2},i)\in(\mathcal{I}^{-}\cup\mathcal{J}^{-})^{c}} & ((z_{-\frac{1}{2},-1,i}-z_{-\frac{1}{2},-1,1})(z_{-\frac{1}{2},-1,i}+z_{-\frac{1}{2},-1,1}+\rho\frac{1}{2}))_{i\in[2:m]}
\end{bmatrix}
\]
of $\hat{G}$ is not of full rank (due to that $p_{s,x,i}>0$ for
all $s,x,i$), we know that $z_{-\frac{1}{2},-1,i}$'s are identical
for all $i\in[m]$. Similarly, the same conclusion holds for other
pairs $(s,x)$.

Therefore, $Z$ is a function of $(S,X)$, or equivalently, can be
expressed in the form of \eqref{eq:-50}.

\section{\label{sec:Proof-of-Theorem-symmetric-1}Proof of Theorem \ref{thm:symmetric-1}}

The Lagrangian of the optimization problem in \eqref{eq:generalboundsimple-2}
is 
\begin{align*}
\mathcal{L}:= & \sum_{s,x,i}p_{s,x,i}\Big\{\Phi(\frac{1}{2}+\rho z_{s,x,i})+\theta_{s,x,i}(\frac{1}{2}+\rho z_{s,x,i})+\theta_{s,x,i}'(1-(\frac{1}{2}+\rho z_{s,x,i}))\\
 & \qquad+\lambda(\rho sz_{s,x,i}-z_{s,x,i}^{2}+(1-\rho)\omega(\beta))+\eta z_{s,x,i}+\eta'(xz_{s,x,i}-\beta)\Big\}\\
 & \qquad+\sum_{s,x}\mu_{s,x}(\sum_{i}p_{s,x,i}-P_{SX}(s,x)).
\end{align*}
Note that the distribution in \eqref{eq:-50} (by choosing proper
$z_{1},z_{2}$ so that it is feasible to the program in \eqref{eq:generalboundsimple-2})
is also feasible in the program \eqref{eq:-49} (specifically, which
corresponds to $p=\frac{1}{4}+\frac{\beta}{2},q=\frac{1}{4}-\frac{\beta}{2}$).
Hence the objective value induced by this solution is no more than
$\overline{\Upsilon}_{\rho}$. From this observation and Statement
4 of Proposition \ref{prop:property-1}, it suffices to consider the
case in which the LICQ is satisfied for some symmetric optimal solution
such that $P_{SXZ}(-s,-x,-z)=P_{SXZ}(s,x,z)$ for all $(s,x,z)$.
We next prove the desired result for this case.

By the KKT theorem, we have the following first-order necessary conditions
for (local) optimal solutions:For even $m\ge8$, 
\begin{align}
 & \frac{\partial\mathcal{L}}{\partial z_{s,x,i}}=p_{s,x,i}[\rho(\Phi'(\frac{1}{2}+\rho z_{s,x,i})+\theta_{s,x,i}-\theta_{s,x,i}')+\lambda(s-2z_{s,x,i})+\eta+\eta'x]=0\label{eq:dLagrangian-1}\\
 & \frac{\partial\mathcal{L}}{\partial p_{s,x,i}}=\Phi(\frac{1}{2}+\rho z_{s,x,i})+\lambda(\rho sz_{s,x,i}-z_{s,x,i}^{2}+(1-\rho)\omega(\beta)) \\
 & \qquad  \qquad   \qquad \qquad \qquad \qquad +\eta z_{s,x,i}+\eta'(xz_{s,x,i}-\beta)+\mu_{s,x}=0\\
 & 0\leq\frac{1}{2}+\rho z_{s,x,i}\leq1,\forall(s,x,i)\in\mathcal{S}\times\mathcal{X}\times[m]\label{eq:-78-2-2-1}\\
 & \sum_{s,x,i}p_{s,x,i}z_{s,x,i}=0,\\
 & \sum_{s,x,i}p_{s,x,i}xz_{s,x,i}=\beta,\\
 & \sum_{s,x,i}p_{s,x,i}(z_{s,x,i}^{2}-\rho sz_{s,x,i})=(1-\rho)\omega(\beta).\\
 & \sum_{i}p_{s,x,i}=P_{SX}(s,x),\forall(s,x)\in\mathcal{S}\times\mathcal{X},\\
 & p_{s,x,i}>0,\forall(s,x,i)\in\mathcal{S}\times\mathcal{X}\times[m]\\
 & \theta_{s,x,i}(\frac{1}{2}+\rho z_{s,x,i})=0\\
 & \theta_{s,x,i}'(\frac{1}{2}-\rho z_{s,x,i})=0\\
 & \lambda\geq0,\overrightarrow{\theta},\overrightarrow{\theta'}\geq\overrightarrow{0}.\label{eq:-17-1}
\end{align}

For $(s,x,i)$ such that $-\frac{1}{2\rho}<z_{s,x,i}<\frac{1}{2\rho}$,
we have $\theta_{s,x,i}=\theta_{s,x,i}'=0$, and hence, 
\begin{align}
\frac{1}{p_{s,x,i}}\frac{\partial\mathcal{L}}{\partial z_{s,x,i}} & =g(s,z_{s,x,i})=0\label{eq:-16-1}\\
\frac{\partial\mathcal{L}}{\partial p_{s,x,i}} & =G(s,z_{s,x,i})+\Phi(\frac{1}{2})+\lambda(1-\rho)\omega(\beta)-\eta'\beta+\mu_{s,x}=0,\label{eq:-6-1}
\end{align}
where 
\begin{align*}
g(s,x,z) & :=\rho\Phi'(\frac{1}{2}+\rho z)+\lambda(s-2z)+\eta+\eta'x,\\
G(s,x,z) & :=\int_{0}^{z}g(s,x,z)\mathrm{d}z.
\end{align*}
Equations \eqref{eq:geq} and \eqref{eq:Geq} imply that given $(s,x)$,
$g(s,x,z)$ is always equal to $0$ and $G(s,x,z)$ remains the same
for all $z\in\{z_{s,x,i}\}_{i\in[m]}$.

By assumption, $\Phi$ is symmetric w.r.t. $\frac{1}{2}$ and $\Phi'$
is increasing, continuous on $(0,1)$, and strictly concave on $(0,\frac{1}{2}]$.
Hence, for each $(s,x)$, $g(s,x,z)=0$ with $z$ unknown has at most
three distinct solutions. By the same argument used in the proof of
Theorem \ref{thm:symmetric}, we have that given $(s,x)$, $z_{i}$
and $z_{i+1}$ cannot be solutions of \eqref{eq:-6-1} at the same
time. Moreover, if given $(s,x)$, the set $\{z_{s,x,i}\}_{i\in[m]}$
contains two distinct solutions of \eqref{eq:-16-1} and \eqref{eq:-6-1},
then in this case, \eqref{eq:-16-1} and \eqref{eq:-6-1} must have
three distinct solutions, and the smallest and the largest among them,
denoted by $z_{1},z_{3}$ respectively, are contained in $\{z_{s,x,i}\}_{i\in[m]}$.
Since $G$ is the integral of $g$ and $G(s,x,z_{1})=G(s,x,z_{3})$
and $\Phi'$ is symmetric w.r.t. $1/2$, we have $z_{1}+z_{3}=0$.
Hence, one of the following three cases could occur for some $\hat{z}_{1},\hat{z}_{2}$.
\\
 Case 1: $\{z_{-\frac{1}{2},-1,i}\}_{i\in[m]}\subseteq\{-\frac{1}{2\rho},\hat{z}_{1}\}$
and $\{z_{\frac{1}{2},-1,i}\}_{i\in[m]}\subseteq\{\hat{z}_{2},\frac{1}{2\rho}\}$.
\\
 Case 2: $\{z_{-\frac{1}{2},-1,i}\}_{i\in[m]}\subseteq\{-\hat{z}_{1},\hat{z}_{1}\}$
and $\{z_{\frac{1}{2},-1,i}\}_{i\in[m]}\subseteq\{\frac{1}{2\rho}\}$.
\\
 Case 3: $\{z_{-\frac{1}{2},-1,i}\}_{i\in[m]}\subseteq\{-\frac{1}{2\rho}\}$
and $\{z_{\frac{1}{2},-1,i}\}_{i\in[m]}\subseteq\{-\hat{z}_{2},\hat{z}_{2}\}$.

Since by assumption, the optimal solution $P_{SXZ}$ is symmetric
in the sense that $P_{SXZ}(-s,-x,-z)=P_{SXZ}(s,x,z)$ for all $(s,x,z)$,
the values of $\{z_{s,1,i}\}_{(s,i)\in\mathcal{S}\times[m]}$ are
determined by $\{z_{s,-1,i}\}_{(s,i)\in\mathcal{S}\times[m]}$. Specifically,
for Case 1, $P_{SXZ}$ can be expressed as 
\[
P_{SXZ}(s,x,z)=\begin{cases}
\frac{1+2\beta}{4}-p & (s,x,z)=\pm(-\frac{1}{2},-1,-\frac{1}{2\rho})\\
p & (s,x,z)=\pm(-\frac{1}{2},-1,\hat{z}_{1})\\
\frac{1-2\beta}{4}-q & (s,x,z)=\pm(\frac{1}{2},-1,\frac{1}{2\rho})\\
q & (s,x,z)=\pm(\frac{1}{2},-1,\hat{z}_{2})
\end{cases}
\]
Solving $\mathbb{E}Z=0,\mathbb{E}[XZ]=\beta$ gives the solution $p,q$
in \eqref{eq:p-2} and \eqref{eq:q-2} with $z_{1}\leftarrow\hat{z}_{1},z_{2}\leftarrow\hat{z}_{2}$.
This solution yields $\overline{\Upsilon}_{\rho}$.

For Case 2, 
\[
P_{SXZ}(s,x,z)=\begin{cases}
\frac{1+2\beta}{4}-p & (s,x,z)=\pm(-\frac{1}{2},-1,-\hat{z}_{1})\\
p & (s,x,z)=\pm(-\frac{1}{2},-1,\hat{z}_{1})\\
\frac{1-2\beta}{4} & (s,x,z)=\pm(\frac{1}{2},-1,\frac{1}{2\rho})
\end{cases}
\]
Solving $\mathbb{E}Z=0,\mathbb{E}[XZ]=\beta$ gives $\hat{z}_{1}=-\frac{4\beta\rho-2\beta+1}{4\beta\rho-16p\rho+2\rho}$.
However, for this case, $\mathbb{E}[Z^{2}]-\rho\mathbb{E}[SZ]\ge(1-\rho)/4$,
which is strictly larger than $(1-\rho)G(\beta)$ for $\beta\in[0,1/2)$.
Moreover, $\mathbb{E}[Z^{2}]-\rho\mathbb{E}[SZ]=(1-\rho)/4$ holds
only if $\beta=p=1/2$. For the case $\beta=p=1/2$, we have $z_{0}=1/2$,
which corresponds to the case in which the Boolean functions are in
fact the dictator functions. This feasible solution is also feasible
in the program \eqref{eq:-49} (specifically, which corresponds to
$\beta=p=1/2,-z_{1}=z_{2}=1/2$). Hence the objective value induced
by this solution is no more than $\overline{\Upsilon}_{\rho}$.

Similarly to Case 2, it can be checked that the objective value for
Case 3 is also no more than $\overline{\Upsilon}_{\rho}$. (In fact,
Case 3 is the same to Case 2 if we replace $\beta\leftarrow-\beta,X\leftarrow-X$
in Case 2. Note that, in this equivalent setting, $\beta$ is nonpositive.)

\backmatter


\bmhead{Acknowledgments}

This research was supported by by the NSFC grant 62101286 and the
Fundamental Research Funds for the Central Universities of China (Nankai
University).

\bmhead{Data availibility} Data sharing is not applicable to this
article as no data were created or analyzed in this study.

\bibliography{ref}


\begin{thebibliography}{44}
\ifx \bisbn   \undefined \def \bisbn  #1{ISBN #1}\fi
\ifx \binits  \undefined \def \binits#1{#1}\fi
\ifx \bauthor  \undefined \def \bauthor#1{#1}\fi
\ifx \batitle  \undefined \def \batitle#1{#1}\fi
\ifx \bjtitle  \undefined \def \bjtitle#1{#1}\fi
\ifx \bvolume  \undefined \def \bvolume#1{\textbf{#1}}\fi
\ifx \byear  \undefined \def \byear#1{#1}\fi
\ifx \bissue  \undefined \def \bissue#1{#1}\fi
\ifx \bfpage  \undefined \def \bfpage#1{#1}\fi
\ifx \blpage  \undefined \def \blpage #1{#1}\fi
\ifx \burl  \undefined \def \burl#1{\textsf{#1}}\fi
\ifx \doiurl  \undefined \def \doiurl#1{\url{https://doi.org/#1}}\fi
\ifx \betal  \undefined \def \betal{\textit{et al.}}\fi
\ifx \binstitute  \undefined \def \binstitute#1{#1}\fi
\ifx \binstitutionaled  \undefined \def \binstitutionaled#1{#1}\fi
\ifx \bctitle  \undefined \def \bctitle#1{#1}\fi
\ifx \beditor  \undefined \def \beditor#1{#1}\fi
\ifx \bpublisher  \undefined \def \bpublisher#1{#1}\fi
\ifx \bbtitle  \undefined \def \bbtitle#1{#1}\fi
\ifx \bedition  \undefined \def \bedition#1{#1}\fi
\ifx \bseriesno  \undefined \def \bseriesno#1{#1}\fi
\ifx \blocation  \undefined \def \blocation#1{#1}\fi
\ifx \bsertitle  \undefined \def \bsertitle#1{#1}\fi
\ifx \bsnm \undefined \def \bsnm#1{#1}\fi
\ifx \bsuffix \undefined \def \bsuffix#1{#1}\fi
\ifx \bparticle \undefined \def \bparticle#1{#1}\fi
\ifx \barticle \undefined \def \barticle#1{#1}\fi
\bibcommenthead
\ifx \bconfdate \undefined \def \bconfdate #1{#1}\fi
\ifx \botherref \undefined \def \botherref #1{#1}\fi
\ifx \url \undefined \def \url#1{\textsf{#1}}\fi
\ifx \bchapter \undefined \def \bchapter#1{#1}\fi
\ifx \bbook \undefined \def \bbook#1{#1}\fi
\ifx \bcomment \undefined \def \bcomment#1{#1}\fi
\ifx \oauthor \undefined \def \oauthor#1{#1}\fi
\ifx \citeauthoryear \undefined \def \citeauthoryear#1{#1}\fi
\ifx \endbibitem  \undefined \def \endbibitem {}\fi
\ifx \bconflocation  \undefined \def \bconflocation#1{#1}\fi
\ifx \arxivurl  \undefined \def \arxivurl#1{\textsf{#1}}\fi
\csname PreBibitemsHook\endcsname

\bibitem[\protect\citeauthoryear{Csisz{\'a}r}{1964}]{csiszar1964informationstheoretische}
\begin{barticle}
\bauthor{\bsnm{Csisz{\'a}r}, \binits{I.}}:
\batitle{Eine informationstheoretische ungleichung und ihre anwendung auf
  beweis der ergodizitaet von markoffschen ketten}.
\bjtitle{Magyer Tud. Akad. Mat. Kutato Int. Koezl.}
\bvolume{8},
\bfpage{85}--\blpage{108}
(\byear{1964})
\end{barticle}
\endbibitem

\bibitem[\protect\citeauthoryear{Csisz{\'a}r}{1967}]{csiszar1967information}
\begin{barticle}
\bauthor{\bsnm{Csisz{\'a}r}, \binits{I.}}:
\batitle{Information-type measures of difference of probability distributions
  and indirect observation}.
\bjtitle{Studia Scientiarum Mathematicarum Hungarica}
\bvolume{2},
\bfpage{229}--\blpage{318}
(\byear{1967})
\end{barticle}
\endbibitem

\bibitem[\protect\citeauthoryear{Ali and Silvey}{1966}]{ali1966general}
\begin{barticle}
\bauthor{\bsnm{Ali}, \binits{S.M.}},
\bauthor{\bsnm{Silvey}, \binits{S.D.}}:
\batitle{A general class of coefficients of divergence of one distribution from
  another}.
\bjtitle{Journal of the Royal Statistical Society: Series B (Methodological)}
\bvolume{28}(\bissue{1}),
\bfpage{131}--\blpage{142}
(\byear{1966})
\end{barticle}
\endbibitem

\bibitem[\protect\citeauthoryear{Tsallis}{1994}]{tsallis1994numbers}
\begin{barticle}
\bauthor{\bsnm{Tsallis}, \binits{C.}}:
\batitle{What are the numbers that experiments provide}.
\bjtitle{Quimica Nova}
\bvolume{17}(\bissue{6}),
\bfpage{468}--\blpage{471}
(\byear{1994})
\end{barticle}
\endbibitem

\bibitem[\protect\citeauthoryear{Eldan}{2015}]{eldan2015two}
\begin{barticle}
\bauthor{\bsnm{Eldan}, \binits{R.}}:
\batitle{A two-sided estimate for the {Gaussian} noise stability deficit}.
\bjtitle{Inventiones Mathematicae}
\bvolume{201}(\bissue{2}),
\bfpage{561}--\blpage{624}
(\byear{2015})
\end{barticle}
\endbibitem

\bibitem[\protect\citeauthoryear{Li and M{\'e}dard}{2019}]{li2019boolean}
\begin{botherref}
\oauthor{\bsnm{Li}, \binits{J.}},
\oauthor{\bsnm{M{\'e}dard}, \binits{M.}}:
Boolean functions: {Noise} stability, non-interactive correlation distillation,
  and mutual information.
IEEE Trans. Inf. Theory
(2019)
\end{botherref}
\endbibitem

\bibitem[\protect\citeauthoryear{Mossel and O'Donnell}{2005}]{mossel2005coin}
\begin{barticle}
\bauthor{\bsnm{Mossel}, \binits{E.}},
\bauthor{\bsnm{O'Donnell}, \binits{R.}}:
\batitle{Coin flipping from a cosmic source: {On} error correction of truly
  random bits}.
\bjtitle{Random Structures \& Algorithms}
\bvolume{26}(\bissue{4}),
\bfpage{418}--\blpage{436}
(\byear{2005})
\end{barticle}
\endbibitem

\bibitem[\protect\citeauthoryear{Sason and Verd{\'u}}{2016}]{sason2016f}
\begin{barticle}
\bauthor{\bsnm{Sason}, \binits{I.}},
\bauthor{\bsnm{Verd{\'u}}, \binits{S.}}:
\batitle{$ f $-divergence inequalities}.
\bjtitle{IEEE Trans. Inf. Theory}
\bvolume{62}(\bissue{11}),
\bfpage{5973}--\blpage{6006}
(\byear{2016})
\end{barticle}
\endbibitem

\bibitem[\protect\citeauthoryear{G{\'a}cs and
  K{\"o}rner}{1973}]{gacs1973common}
\begin{barticle}
\bauthor{\bsnm{G{\'a}cs}, \binits{P.}},
\bauthor{\bsnm{K{\"o}rner}, \binits{J.}}:
\batitle{Common information is far less than mutual information}.
\bjtitle{Problems of Control and Information Theory}
\bvolume{2}(\bissue{2}),
\bfpage{149}--\blpage{162}
(\byear{1973})
\end{barticle}
\endbibitem

\bibitem[\protect\citeauthoryear{Witsenhausen}{1975}]{witsenhausen1975sequences}
\begin{barticle}
\bauthor{\bsnm{Witsenhausen}, \binits{H.S.}}:
\batitle{On sequences of pairs of dependent random variables}.
\bjtitle{SIAM Journal on Applied Mathematics}
\bvolume{28}(\bissue{1}),
\bfpage{100}--\blpage{113}
(\byear{1975})
\end{barticle}
\endbibitem

\bibitem[\protect\citeauthoryear{Kumar and Courtade}{2013}]{kumar2013boolean}
\begin{bchapter}
\bauthor{\bsnm{Kumar}, \binits{G.R.}},
\bauthor{\bsnm{Courtade}, \binits{T.A.}}:
\bctitle{Which {Boolean} functions are most informative?}
In: \bbtitle{2013 IEEE International Symposium on Information Theory},
pp. \bfpage{226}--\blpage{230}
(\byear{2013}).
\bcomment{IEEE}
\end{bchapter}
\endbibitem

\bibitem[\protect\citeauthoryear{Courtade and
  Kumar}{2014}]{courtade2014boolean}
\begin{barticle}
\bauthor{\bsnm{Courtade}, \binits{T.A.}},
\bauthor{\bsnm{Kumar}, \binits{G.R.}}:
\batitle{Which {Boolean} functions maximize mutual information on noisy
  inputs?}
\bjtitle{IEEE Trans. Inf. Theory}
\bvolume{60}(\bissue{8}),
\bfpage{4515}--\blpage{4525}
(\byear{2014})
\end{barticle}
\endbibitem

\bibitem[\protect\citeauthoryear{Anantharam et~al.}{2013}]{anantharam2013on}
\begin{bchapter}
\bauthor{\bsnm{Anantharam}, \binits{V.}},
\bauthor{\bsnm{Gohari}, \binits{A.A.}},
\bauthor{\bsnm{Kamath}, \binits{S.}},
\bauthor{\bsnm{Nair}, \binits{C.}}:
\bctitle{On hypercontractivity and the mutual information between {Boolean}
  functions}.
In: \bbtitle{Communication, Control, and Computing (Allerton), 2013 51th Annual
  Allerton Conference On},
pp. \bfpage{13}--\blpage{19}
(\byear{2013}).
\bcomment{IEEE}
\end{bchapter}
\endbibitem

\bibitem[\protect\citeauthoryear{Kindler et~al.}{2015}]{kindler2015remarks}
\begin{botherref}
\oauthor{\bsnm{Kindler}, \binits{G.}},
\oauthor{\bsnm{O'Donnell}, \binits{R.}},
\oauthor{\bsnm{Witmer}, \binits{D.}}:
Remarks on the most informative function conjecture at fixed mean.
arXiv preprint arXiv:1506.03167
(2015)
\end{botherref}
\endbibitem

\bibitem[\protect\citeauthoryear{Ordentlich
  et~al.}{2016}]{ordentlich2016improved}
\begin{bchapter}
\bauthor{\bsnm{Ordentlich}, \binits{O.}},
\bauthor{\bsnm{Shayevitz}, \binits{O.}},
\bauthor{\bsnm{Weinstein}, \binits{O.}}:
\bctitle{An improved upper bound for the most informative {Boolean} function
  conjecture}.
In: \bbtitle{2016 IEEE International Symposium on Information Theory (ISIT)},
pp. \bfpage{500}--\blpage{504}
(\byear{2016}).
\bcomment{IEEE}
\end{bchapter}
\endbibitem

\bibitem[\protect\citeauthoryear{Samorodnitsky}{2016}]{samorodnitsky2016entropy}
\begin{barticle}
\bauthor{\bsnm{Samorodnitsky}, \binits{A.}}:
\batitle{On the entropy of a noisy function}.
\bjtitle{IEEE Trans. Inf. Theory}
\bvolume{62}(\bissue{10}),
\bfpage{5446}--\blpage{5464}
(\byear{2016})
\end{barticle}
\endbibitem

\bibitem[\protect\citeauthoryear{Pichler et~al.}{2018}]{pichler2018dictator}
\begin{barticle}
\bauthor{\bsnm{Pichler}, \binits{G.}},
\bauthor{\bsnm{Piantanida}, \binits{P.}},
\bauthor{\bsnm{Matz}, \binits{G.}}:
\batitle{Dictator functions maximize mutual information}.
\bjtitle{The Annals of Applied Probability}
\bvolume{28}(\bissue{5}),
\bfpage{3094}--\blpage{3101}
(\byear{2018})
\end{barticle}
\endbibitem

\bibitem[\protect\citeauthoryear{Barnes and
  {\"O}zg{\"u}r}{2020}]{barnes2020courtade}
\begin{bchapter}
\bauthor{\bsnm{Barnes}, \binits{L.P.}},
\bauthor{\bsnm{{\"O}zg{\"u}r}, \binits{A.}}:
\bctitle{The {Courtade-Kumar} most informative {Boolean} function conjecture
  and a symmetrized {Li-M\'edard} conjecture are equivalent}.
In: \bbtitle{2020 IEEE International Symposium on Information Theory (ISIT)},
pp. \bfpage{2205}--\blpage{2209}
(\byear{2020}).
\bcomment{IEEE}
\end{bchapter}
\endbibitem

\bibitem[\protect\citeauthoryear{Friedgut et~al.}{2002}]{friedgut2002boolean}
\begin{barticle}
\bauthor{\bsnm{Friedgut}, \binits{E.}},
\bauthor{\bsnm{Kalai}, \binits{G.}},
\bauthor{\bsnm{Naor}, \binits{A.}}:
\batitle{Boolean functions whose fourier transform is concentrated on the first
  two levels}.
\bjtitle{Advances in Applied Mathematics}
\bvolume{29}(\bissue{3}),
\bfpage{427}--\blpage{437}
(\byear{2002})
\end{barticle}
\endbibitem

\bibitem[\protect\citeauthoryear{Borell}{1985}]{borell1985geometric}
\begin{barticle}
\bauthor{\bsnm{Borell}, \binits{C.}}:
\batitle{Geometric bounds on the {Ornstein--Uhlenbeck} velocity process}.
\bjtitle{Probability Theory and Related Fields}
\bvolume{70}(\bissue{1}),
\bfpage{1}--\blpage{13}
(\byear{1985})
\end{barticle}
\endbibitem

\bibitem[\protect\citeauthoryear{Benjamini et~al.}{1999}]{benjamini1999noise}
\begin{barticle}
\bauthor{\bsnm{Benjamini}, \binits{I.}},
\bauthor{\bsnm{Kalai}, \binits{G.}},
\bauthor{\bsnm{Schramm}, \binits{O.}}:
\batitle{Noise sensitivity of {Boolean} functions and applications to
  percolation}.
\bjtitle{Publications Math{\'e}matiques de l'Institut des Hautes Etudes
  Scientifiques}
\bvolume{90}(\bissue{1}),
\bfpage{5}--\blpage{43}
(\byear{1999})
\end{barticle}
\endbibitem

\bibitem[\protect\citeauthoryear{O'Donnell}{2014}]{O'Donnell14analysisof}
\begin{bbook}
\bauthor{\bsnm{O'Donnell}, \binits{R.}}:
\bbtitle{Analysis of {Boolean} Functions}.
\bpublisher{Cambridge University Press}, \blocation{???}
(\byear{2014})
\end{bbook}
\endbibitem

\bibitem[\protect\citeauthoryear{Witsenhausen and
  Wyner}{1975}]{witsenhausen1975conditional}
\begin{barticle}
\bauthor{\bsnm{Witsenhausen}, \binits{H.}},
\bauthor{\bsnm{Wyner}, \binits{A.}}:
\batitle{A conditional entropy bound for a pair of discrete random variables}.
\bjtitle{IEEE Trans. Inf. Theory}
\bvolume{21}(\bissue{5}),
\bfpage{493}--\blpage{501}
(\byear{1975})
\end{barticle}
\endbibitem

\bibitem[\protect\citeauthoryear{Erkip}{1996}]{erkip1996efficiency}
\begin{botherref}
\oauthor{\bsnm{Erkip}, \binits{E.}}:
The efficiency of information in investment.
PhD thesis,
Ph.D. dissertation, Dept. Electr. Eng., Stanford Univ. Press, Stanford, CA, USA
(1996)
\end{botherref}
\endbibitem

\bibitem[\protect\citeauthoryear{Wyner and Ziv}{1973}]{wyner1973theorem}
\begin{barticle}
\bauthor{\bsnm{Wyner}, \binits{A.D.}},
\bauthor{\bsnm{Ziv}, \binits{J.}}:
\batitle{A theorem on the entropy of certain binary sequences and applications:
  {Part I}}.
\bjtitle{IEEE Trans. Inf. Theory}
\bvolume{19}(\bissue{6}),
\bfpage{769}--\blpage{772}
(\byear{1973})
\end{barticle}
\endbibitem

\bibitem[\protect\citeauthoryear{Yang and Wesel}{2019}]{yang2019most}
\begin{bchapter}
\bauthor{\bsnm{Yang}, \binits{H.}},
\bauthor{\bsnm{Wesel}, \binits{R.D.}}:
\bctitle{On the most informative boolean functions of the very noisy channel}.
In: \bbtitle{2019 IEEE International Symposium on Information Theory (ISIT)},
pp. \bfpage{1202}--\blpage{1206}
(\byear{2019}).
\bcomment{IEEE}
\end{bchapter}
\endbibitem

\bibitem[\protect\citeauthoryear{Eldan et~al.}{2022}]{eldan2022noise}
\begin{botherref}
\oauthor{\bsnm{Eldan}, \binits{R.}},
\oauthor{\bsnm{Mikulincer}, \binits{D.}},
\oauthor{\bsnm{Raghavendra}, \binits{P.}}:
Noise stability on the {Boolean} hypercube via a renormalized {Brownian}
  motion.
arXiv preprint arXiv:2208.06508
(2022)
\end{botherref}
\endbibitem

\bibitem[\protect\citeauthoryear{Yu and Tan}{2021}]{yu2021non}
\begin{barticle}
\bauthor{\bsnm{Yu}, \binits{L.}},
\bauthor{\bsnm{Tan}, \binits{V.Y.F.}}:
\batitle{On non-interactive simulation of binary random variables}.
\bjtitle{IEEE Trans. Inf. Theory}
\bvolume{67}(\bissue{4}),
\bfpage{2528}--\blpage{2538}
(\byear{2021})
\end{barticle}
\endbibitem

\bibitem[\protect\citeauthoryear{Kahn et~al.}{1988}]{kahn1988influence}
\begin{bchapter}
\bauthor{\bsnm{Kahn}, \binits{J.}},
\bauthor{\bsnm{Kalai}, \binits{G.}},
\bauthor{\bsnm{Linial}, \binits{N.}}:
\bctitle{The influence of variables on {Boolean} functions}.
In: \bbtitle{29th Annual Symposium on Foundations of Computer Science},
pp. \bfpage{68}--\blpage{80}
(\byear{1988}).
\bcomment{IEEE}
\end{bchapter}
\endbibitem

\bibitem[\protect\citeauthoryear{Mossel et~al.}{2006}]{mossel2006non}
\begin{barticle}
\bauthor{\bsnm{Mossel}, \binits{E.}},
\bauthor{\bsnm{O'Donnell}, \binits{R.}},
\bauthor{\bsnm{Regev}, \binits{O.}},
\bauthor{\bsnm{Steif}, \binits{J.E.}},
\bauthor{\bsnm{Sudakov}, \binits{B.}}:
\batitle{Non-interactive correlation distillation, inhomogeneous {Markov}
  chains, and the reverse {Bonami-Beckner} inequality}.
\bjtitle{Israel Journal of Mathematics}
\bvolume{154}(\bissue{1}),
\bfpage{299}--\blpage{336}
(\byear{2006})
\end{barticle}
\endbibitem

\bibitem[\protect\citeauthoryear{Kamath and Anantharam}{2016}]{kamath2016non}
\begin{barticle}
\bauthor{\bsnm{Kamath}, \binits{S.}},
\bauthor{\bsnm{Anantharam}, \binits{V.}}:
\batitle{On non-interactive simulation of joint distributions}.
\bjtitle{IEEE Trans. Inf. Theory}
\bvolume{62}(\bissue{6}),
\bfpage{3419}--\blpage{3435}
(\byear{2016})
\end{barticle}
\endbibitem

\bibitem[\protect\citeauthoryear{Ordentlich et~al.}{2020}]{ordentlich2020note}
\begin{barticle}
\bauthor{\bsnm{Ordentlich}, \binits{O.}},
\bauthor{\bsnm{Polyanskiy}, \binits{Y.}},
\bauthor{\bsnm{Shayevitz}, \binits{O.}}:
\batitle{A note on the probability of rectangles for correlated binary
  strings}.
\bjtitle{IEEE Trans. Inf. Theory}
\bvolume{66}(\bissue{12}),
\bfpage{7878}--\blpage{7886}
(\byear{2020})
\end{barticle}
\endbibitem

\bibitem[\protect\citeauthoryear{Kirshner and
  Samorodnitsky}{2019}]{kirshner2019moment}
\begin{botherref}
\oauthor{\bsnm{Kirshner}, \binits{N.}},
\oauthor{\bsnm{Samorodnitsky}, \binits{A.}}:
A moment ratio bound for polynomials and some extremal properties of
  {Krawchouk} polynomials and {Hamming} spheres.
arXiv preprint arXiv:1909.11929
(2019)
\end{botherref}
\endbibitem

\bibitem[\protect\citeauthoryear{Yu et~al.}{2021}]{yu2021graphs}
\begin{botherref}
\oauthor{\bsnm{Yu}, \binits{L.}},
\oauthor{\bsnm{Anantharam}, \binits{V.}},
\oauthor{\bsnm{Chen}, \binits{J.}}:
Graphs of joint types, noninteractive simulation, and stronger
  hypercontractivity.
arXiv preprint arXiv:2102.00668
(2021)
\end{botherref}
\endbibitem

\bibitem[\protect\citeauthoryear{Yu}{2021}]{yu2021strong}
\begin{botherref}
\oauthor{\bsnm{Yu}, \binits{L.}}:
Strong {Brascamp-Lieb} inequalities.
arXiv preprint arXiv:2102.06935
(2021)
\end{botherref}
\endbibitem

\bibitem[\protect\citeauthoryear{Bunt}{1934}]{bunt1934bijdrage}
\begin{bbook}
\bauthor{\bsnm{Bunt}, \binits{L.N.H.}}:
\bbtitle{Bijdrage Tot de Theorie der Convexe Puntverzamelingen}.
\bpublisher{Rijksuniversiteit te Groningen}, \blocation{???}
(\byear{1934})
\end{bbook}
\endbibitem

\bibitem[\protect\citeauthoryear{Eggleston}{1966}]{eggleston1966convexity}
\begin{botherref}
\oauthor{\bsnm{Eggleston}, \binits{H.G.}}:
Convexity.
Oxford University Press
(1966)
\end{botherref}
\endbibitem

\bibitem[\protect\citeauthoryear{Peterson}{1973}]{peterson1973review}
\begin{barticle}
\bauthor{\bsnm{Peterson}, \binits{D.W.}}:
\batitle{A review of constraint qualifications in finite-dimensional spaces}.
\bjtitle{Siam Review}
\bvolume{15}(\bissue{3}),
\bfpage{639}--\blpage{654}
(\byear{1973})
\end{barticle}
\endbibitem

\bibitem[\protect\citeauthoryear{Bazaraa et~al.}{2013}]{bazaraa2013nonlinear}
\begin{bbook}
\bauthor{\bsnm{Bazaraa}, \binits{M.S.}},
\bauthor{\bsnm{Sherali}, \binits{H.D.}},
\bauthor{\bsnm{Shetty}, \binits{C.M.}}:
\bbtitle{Nonlinear Programming: Theory and Algorithms}.
\bpublisher{John Wiley \& Sons}, \blocation{???}
(\byear{2013})
\end{bbook}
\endbibitem

\bibitem[\protect\citeauthoryear{Fu et~al.}{2001}]{fu2001minimum}
\begin{barticle}
\bauthor{\bsnm{Fu}, \binits{F.-W.}},
\bauthor{\bsnm{Wei}, \binits{V.K.}},
\bauthor{\bsnm{Yeung}, \binits{R.W.}}:
\batitle{On the minimum average distance of binary codes: {Linear} programming
  approach}.
\bjtitle{Discrete Applied Mathematics}
\bvolume{111}(\bissue{3}),
\bfpage{263}--\blpage{281}
(\byear{2001})
\end{barticle}
\endbibitem

\bibitem[\protect\citeauthoryear{Yu and Tan}{2019}]{yu2019improved}
\begin{botherref}
\oauthor{\bsnm{Yu}, \binits{L.}},
\oauthor{\bsnm{Tan}, \binits{V.Y.F.}}:
An improved linear programming bound on the average distance of a binary code.
arXiv preprint arXiv:1910.09416
(2019)
\end{botherref}
\endbibitem

\bibitem[\protect\citeauthoryear{Chang}{2002}]{chang2002polynomial}
\begin{barticle}
\bauthor{\bsnm{Chang}, \binits{M.-C.}}:
\batitle{A polynomial bound in {Freiman's} theorem}.
\bjtitle{Duke mathematical journal}
\bvolume{113}(\bissue{3}),
\bfpage{399}--\blpage{419}
(\byear{2002})
\end{barticle}
\endbibitem

\bibitem[\protect\citeauthoryear{Jendrej et~al.}{2015}]{jendrej2015some}
\begin{barticle}
\bauthor{\bsnm{Jendrej}, \binits{J.}},
\bauthor{\bsnm{Oleszkiewicz}, \binits{K.}},
\bauthor{\bsnm{Wojtaszczyk}, \binits{J.O.}}:
\batitle{On some extensions of the {FKN} theorem}.
\bjtitle{Theory of Computing}
\bvolume{11}(\bissue{1}),
\bfpage{445}--\blpage{469}
(\byear{2015})
\end{barticle}
\endbibitem

\bibitem[\protect\citeauthoryear{K{\"o}nig et~al.}{1999}]{konig1999projection}
\begin{barticle}
\bauthor{\bsnm{K{\"o}nig}, \binits{H.}},
\bauthor{\bsnm{Sch{\"u}tt}, \binits{C.}},
\bauthor{\bsnm{Tomczak-Jaegermann}, \binits{N.}}:
\batitle{Projection constants of symmetric spaces and variants of
  {Khintchine's} inequality}.
\bjtitle{Journal f{\"u}r die reine und angewandte Mathematik}
\bvolume{1999}(\bissue{511}),
\bfpage{1}--\blpage{42}
(\byear{1999})
\end{barticle}
\endbibitem

\end{thebibliography}
\end{document}